\documentclass{article}
\usepackage{amsmath, amssymb, latexsym, mathrsfs, amsthm, youngtab, enumerate}
\usepackage[all]{xy}
\usepackage[normalem]{ulem}
\usepackage[UKenglish]{babel}
\newtheorem{thm}{Theorem}

\newtheorem{prop}[thm]{Proposition}
\newtheorem{lem}[thm]{Lemma}
\newtheorem{exm}{Example}
\newtheorem{cor}[thm]{Corollary}
\newtheorem*{thmm}{Theorem 2}

\begin{document}

\newcommand{\mfI}{\mathfrak{I}} \newcommand{\mmu}{m_{\mu}} \newcommand{\mnu}{m_{\nu}}
\newcommand{\mfs}{\mathfrak{s}} \newcommand{\ulam}{u^+_{\lambda}} \newcommand{\umu}{u^+_{\mu}}
\newcommand{\mft}{\mathfrak{t}} \newcommand{\domlam} {\sum_{i=1}^{s-1} |\lambda^{(i)}|+\sum_{j=1}^k\lambda^{(s)}_j}
\newcommand{\mttA}{\mathtt{A}}  \newcommand{\dommu}{\sum_{i=1}^{s-1}|\mu^{(i)}|+\sum_{j=1}^k\mu^{(s)}_{j}}
\newcommand{\mttS}{\mathtt{S}}  \newcommand{\ddt}{\mathfrak{d}^{(s)}_{d,t}}
\newcommand{\mttT}{\mathtt{T}}  \newcommand{\ls}{\mathfrak{l}^{(s')}}
\newcommand{\bigH}{\check{\mathscr{H}}}
\newcommand{\AKA}{\mathscr{H}}
\newcommand{\IWA}{\mathscr{H}^{\mathrm{fin}}_n}
\newcommand{\mlam}{m_{\lambda}}

\newcommand{\ten}{10}
\newcommand{\eleven}{11}
\newcommand{\twelve}{12}
\newcommand{\thirteen}{13}
\newcommand{\fourteen}{14}
\newcommand{\fifteen}{15}
\newcommand{\sixteen}{16}
\newcommand{\seventeen}{17}
\newcommand{\eighteen}{18}
\newcommand{\nineteen}{19}
\newcommand{\twenty}{20}

\newcommand{\oneone}{1_1}
\newcommand{\twoone}{2_1}
\newcommand{\threeone}{3_1}
\newcommand{\fourone}{4_1}
\newcommand{\fiveone}{5_1}

\newcommand{\onetwo}{1_2}
\newcommand{\twotwo}{2_2}
\newcommand{\threetwo}{3_2}
\newcommand{\fourtwo}{4_2}
\newcommand{\fivetwo}{5_2}

\newcommand{\onethree}{1_3}
\newcommand{\twothree}{2_3}
\newcommand{\threethree}{3_3}
\newcommand{\fourthree}{4_3}
\newcommand{\fivethree}{5_3}

\title{Homomorphisms between Specht Modules of the Ariki-Koike algebra }
\author{Kelvin Corlett
\thanks{The author wishes to thank {\it Rheynn Ynsee, Reiltys Ellan Vannin} for their support throughout his PhD.}\\
School of Mathematics, University of East Anglia,\\
Norwich, United Kingdom, NR4 7TJ,
\\ \texttt{k.corlett@uea.ac.uk}}

\date{\today}
\maketitle

\begin{abstract}
    In this paper we generalize a theorem due to Lyle, extending its application to the setting of the Ariki-Koike algebra, and in doing so establish an analogue of the kernel intersection theorem.  This in turn provides us with a means towards constructing homomorphisms between Specht modules for the Ariki-Koike algebra. 
\end{abstract}

\tableofcontents

\section{Introduction}

Given positive integers $r$ and $n$, an Ariki-Koike algebra can be defined as a deformation of the group algebra of the  complex reflection group $C_r\wr\mathfrak{S}_n$.  Also known as the (cyclotomic) Hecke algebra of type $G(r,1,n)$, this family of algebras contains the Iwahori-Hecke algebras of type $A$ and type $B$, which correspond to when $r=1$ and $r=2$ respectively. 

In this paper we generalize a result due to Lyle  \cite[Theorem 2.3]{Lyle2} from the representation theory of the Iwahori-Hecke algebra of type $A$, which itself performs a similar role as does  the Kernel Intersection in \cite{DJ1}. This then  provides us with a criterion which when satisfied allows for the explicit combinatorial construction of homomorphisms between the  Specht modules of the Ariki-Koike algebra.

In practice, the construction of these homomorphisms between arbitrary Specht modules is difficult.  However, one case in which this is relatively straightforward is when the diagrams of the multipartitions indexing the given Specht modules differ only in so much as one can be formed from the other by deleting a removable node in one component and inserting an addable node in the preceding component.  When $r=2$ this amounts to a generalization of the `one node homomorphisms' of \cite{Lyle1} and \cite{Lyle2} to the setting of the Iwahori-Hecke Algebra of type $B$. We  cover this case fully in \cite{Cor1}.

Unfortunately,  in this setting we at present lack  a semistandard basis theorem analogous to that found in \cite{DJ1}.  As such we cannot say precisely when we can construct the entire homomorphism space between two given Specht modules by the scheme set out in this paper.  

As a final remark, most of the definitions and notation which form the background to this  paper follows that used by Mathas in his survey regarding the representation theory of the Ariki-Koike and cyclotomic $q$-Schur algebras \cite{Mathas2}.

\subsection{Background}

\subsection{The Ariki-Koike Algebra}

Let $\mathbb{F}$ be a field and let $q\neq1$  be some non-zero element of $\mathbb{F}$. For each pair of positive integers $n$ and $r$ the Ariki-Koike algebra, which we denote by $\mathscr{H}=\mathscr{H}_{r,n}$, is defined via parameters $q$ and $Q_1,Q_2,\ldots,Q_r\in\mathbb{F}$ as the unital associative algebra with generators $T_0,T_1,\ldots,T_{n-1} $, subject to the following relations:
\begin{align*}
\left(T_0-Q_1\right)\left(T_0-Q_2\right)\cdots\left(T_0-Q_r\right)&=0\\
T_0T_1T_0T_1&=T_1T_0T_1T_0\\\left(T_i-q\right)\left(T_i+1\right)&=0 &\text{for }1\leq i\leq n-1,\\T_iT_{i+1}T_i&=T_{i+1}T_iT_{i+1} &\text{for }1\leq i<n-1,\\T_iT_j&=T_jT_i &\text{for }0\leq i<j-1<n-1.
\end{align*}
The condition that $q\neq1$ is necessary because otherwise  the corresponding theory demands a `degenerate' version of the Ariki-Koike algebra, which we do not go into here.     
 
If $w=s_{i_1}s_{i_2}\cdots s_{i_k}\in\mathfrak{S}_n$ is a reduced expression, where $s_i$ is the simple transposition $(i,i+1)$, then we define $T_{w}=T_{i_1}T_{i_2}\cdots T_{i_k}$.  We will usually write $T_{i_1}T_{i_2}\cdots T_{i_k}$ as $T_{i_1,i_2,\ldots,i_k}$. 

 It is clear that the subalgebra generated by $T_{1},T_{2},\ldots ,T_{n-1}$ is isomorphic to the Iwahori-Hecke algebra of type $\mathrm{A}$.  This fact will allow us to make use of its representation theory in establishing a number of  the results which follow. 

For each $1\leq k\leq n$ we define elements \begin{displaymath}L_k:=q^{1-k}T_{k-1}T_{k-2}\cdots T_1T_0T_1\cdots T_{k-2}T_{k-1}.\end{displaymath}
These elements are an analogue of the Jucys-Murphy elements of the Iwahori-Hecke algebra of type $A$ and of the group algebra of the symmetric group.

  Some of the properties of these elements which will repeatedly prove useful are summarised in the following proposition (for further details, see \cite{DJM}) : 
\begin{prop}\label{Jucys}\mbox{}\begin{enumerate}\item The elements $L_i$ generate an abelian subalgebra of $\mathscr{H}$;
\item Let $l$ be some non-negative integer.  If $j\neq k$ then $T_j$ and $\prod^{k}_{i=1}(L_i-Q_l)$ commute; and
\item $L_iT_j=T_jL_i$ whenever $j\neq i, i-1$.  
\end{enumerate}
\end{prop}
\paragraph{Multipartitions and tableaux}
Recall that a composition of $n$ is a sequence $\lambda=\lambda_1,\lambda_2,\ldots$ of non-negative integers such that $|\lambda|=\sum_i\lambda_i=n$.  We say that $\lambda$ is a partition if it satisfies the additional condition that $\lambda_{i}\geq\lambda_{i+1}$ for all $i\geq1$.     

A multicomposition $\mu$ of $n$ in $r$ parts is a sequence of compositions $\mu^{(i)}$ for $1\leq i\leq r$, written
\begin{displaymath} 
\mu=\left(\mu^{(1)},\mu^{(2)},\ldots,\mu^{(r)}\right),
\end{displaymath} such that $|\mu|=\sum|\mu^{(i)}|=n$.  For every $1\leq i\leq r$,  the composition $\mu^{(i)}$ is the $i$-th component or part of $\mu$.  Additionally,   a multicomposition in which each component is a partition is a multipartition of $n$ in $r$ parts.  For brevity, we will often refer to these as `$r$-multicompositions' or multipartitions of $n$, or simply as multicompositions and multipartitions.   

The set of multipartitions and multicompositions of $n$ in $r$ parts can be partially ordered under the dominance relation; that is, $\mu$ is said to dominate $\lambda$, in which case we write $\lambda\unlhd\mu$, if and only if
\begin{displaymath}
\sum^{l-1}_{i=0}|\lambda^{(i)}|+\sum_{j=0}^s \lambda^{(l)}_j\leq\sum^{l-1}_{i=0}|\mu^{(i)}|+\sum_{j=0}^s\mu^{(l)}_j
\end{displaymath}
for all positive integers $l$ and $s$.

As is the case for compositions, we can represent  multicompositions via  Young diagrams; the diagram of a multicomposition $\mu$ being a sequence of Young diagrams, each  corresponding to the successive compositions which form the components of $\mu$.  More formally, the diagram $[\mu]$ of  a multicomposition $\mu$ is defined thus
\begin{displaymath}
\left[\mu\right]=\left\{\left(i,j,k\right)\in\mathbb{N}^2\times\left\{1,2,\ldots,r\right\}:\lambda^{(k)}_i\leq j\right\}.
\end{displaymath}
Each element of $[\mu]$ is called a node of the diagram. 

Given a multicomposition $\lambda$ we define a $\lambda$-tableau to be a bijection $\mathfrak{t}:[\lambda]\rightarrow\{1,2,\ldots,n\}$.  Such a tableau may be viewed as filling the nodes of $[\lambda]$ with entries taken from $\{1,2,\ldots,n\}$.  For each composition or partition $\lambda$, a $\lambda$-tableau $\mathfrak{t}$ is row standard if its entries are strictly increasing along the rows of each component, and standard if $\lambda$ is a multipartition and, in addition to being row standard, its entries are strictly  increasing down the columns of every component.

 We use $\mathfrak{t}^{\lambda}$ to distinguish the unique $\lambda$-tableau where the entries appear in order along the rows.   If $\mathfrak{t}$ is any other $\lambda$-tableau, let $d(\mathfrak{t})$ denote the unique permutation such that $\mathfrak{t}=\mathfrak{t}^{\lambda}\cdot d(\mathfrak{t})$. Here the action of $\mathfrak{S}_n$ on the set of tableaux is the obvious permutation of entries.     
\begin{exm}
Suppose $\lambda=((2,2,1),(3,1))$, then 
\begin{displaymath}
\mathfrak{t}^\lambda=\left(\,\Yvcentermath1\young(12,34,5)\,,\,\Yvcentermath1\young(678,9)\,\right).
\end{displaymath}  

If $\mathfrak{t}$ is given by 
\begin{displaymath}
\mathfrak{t}=\left(\,\Yvcentermath1\young(13,25,9)\,,\,\Yvcentermath1\young(478,6)\,\right),
\end{displaymath}
then $d(\mathfrak{t})=(2,3)(4,5,9,6)$.
\end{exm} 

\subsubsection{Specht Modules and Semi-standard Homomorphisms}For each multicomposition $\lambda$, let $x_{\lambda}$ be the element of $\mathscr{H}$ which is the sum of all $T_w$ as  $w\in\mathfrak{S}_n$ ranges over all permutations which stabilize the rows of $\mathfrak{t}^{\lambda}$.  Also, let  $u^+_{\lambda}$ be given by
\begin{displaymath}
 u^+_{\lambda}=\prod_{s=2}^r\prod_{i=1}^{|\lambda^{(1)}|+\cdots+|\lambda^{(s-1)}|}(L_i-Q_{s}).
 \end{displaymath}
 Note that $x_{\lambda}$ and $u^+_{\lambda}$ commute with one another.  

We set $m_{\lambda}=x_{\lambda}u^+_{\lambda}=u^+_\lambda x_\lambda$ and define $M^{\lambda}$ to be the right $\mathscr{H}$-module generated by $m_{\lambda}$.  These modules are in some sense an analogue of the permutation modules for the group algebra of $\mathfrak{S}_{n}$, and we will informally refer to them as such.  

If we define $*:\mathscr{H}\rightarrow\mathscr{H}$ to the anti-isomorphism  defined by $T_w\mapsto T_{w^{-1}}$, there is a cellular $\mathbb{F}$-basis of $\mathscr{H}$ provided by
\begin{displaymath}
\left\{m_{\mathfrak{st}}=T^*_{d(\mathfrak{s})}m_\lambda T_{d(\mathfrak{t})}:\lambda \text{ an $r$-multipartition of $n$}, \mathfrak{s},\mathfrak{t}\in\mathrm{Std}(\lambda)\right\}
\end{displaymath}   
where $\mathrm{Std}(\lambda)$ is the set of standard $\lambda$-multipartitions of $n$ in $r$ parts.

\paragraph{Specht Modules}For each multipartition $\lambda$ of $n$ in $r$ parts, let $\check{\mathscr{H}}^{\lambda}$ be the two-sided ideal of $\mathscr{H}$ which has an $\mathbb{F}$-basis
\begin{displaymath}
\left\{m_{\mathfrak{st}}:\mathfrak{s},\mathfrak{t}\in\mathrm{Std}(\mu), \lambda\lhd\mu\right\}.
\end{displaymath} 

The Specht modules are the cell modules of the Ariki-Koike algebra, relative to the cellular basis given above.  Indexed by the $r$-multipartitions  of $n$, the Specht module $S^\lambda$ is the right $\mathscr{H}$-module generated by  $\check{\mathscr{H}}^{\lambda}+m_{\lambda}$.  As an $\mathbb{F}$ module the set
\begin{displaymath}
\left\{\check{\mathscr{H}}^\lambda +m_{\mathfrak{t}^\lambda \mathfrak{t}}:\mathfrak{t}\in\mathrm{Std}(\lambda) \right\}
\end{displaymath}
forms a basis, the elements of which we will write as $m_\mathfrak{t}$, where $m_{\mathfrak{t}}=\check{\mathscr{H}}^\lambda+m_{\mathfrak{t}^\lambda\mathfrak{t}}$ for $\mathfrak{t}\in\mathrm{Std}(\lambda)$.  
\paragraph{Semistandard Tableaux}If $\lambda$ and $\mu$ are multipartitions, then a $\mu$-tableau of type $\lambda$  is a mapping 
\begin{displaymath}\mathtt{T}:[\mu]\rightarrow\{(i,k): i\geq 1, \mathrm{and}\, 1\leq k\leq r\}\end{displaymath} 
such that the number of entries in $\mathtt{T}$ of the form $(i,k)$ for  given values of $i$ and $k$ is the same as the number of nodes in $[\lambda]$ of the form $(i,j,k)$ for some $j$.

  We impose a partial ordering $\preceq$ on the entries of $\mathtt{T}$ by specifying that 

\begin{displaymath}(i,k)\preceq (i',k') \text{ if and only if }k< k', \text{ or } k=k' \text{ and } i\leq i'\end{displaymath}
The purpose of this ordering being to allow the definition of a semistandard tableau, that being a $\mu$-tableau $\mathtt{T}$ of type $\lambda$ such that:

\begin{enumerate}
\item the entries in each row are non-decreasing in each component, relative to the partial ordering just defined;\item the entries are strictly increasing down each column of every component; and \item if $(a,b,c)\in[\mu]$ and $\mathtt{T}(a,b,c)=(i,k)$, then $c\leq k$.  
\end{enumerate}
We will also refer to row-semistandard $\mu$-tableau of type $\lambda$, which we define to be $\mu$-tableau of type $\lambda$ in which the entries are weakly increasing along the rows of each component. We denote the set of semistandard $\mu$-tableaux of type $\lambda$ by $\mathcal{T}_0(\mu,\lambda)$, and the set of row-semistandard $\mu$-tableaux of type $\lambda$ by $\mathcal{T}_r(\mu,\lambda)$.

If $\mathfrak{t}$ is a $\mu$-tableau, let $\lambda(\mathfrak{t})$ be the $\mu$-tableau of type $\lambda$ in which the node $(i,j,k)\in[\mu]$ is occupied by $(x,y)$ if the entry $\mathfrak{t}(i,j,k)$ appears in row $i$ of component $j$ of $\mathfrak{t}^\lambda$. We denote $\lambda(\mathfrak{t}^{\lambda})$ by $\mathtt{T}^\lambda$.

\begin{exm}
If $\lambda=((3,2,1),(2,1))$, $\mu=((4,3,1),(1))$, and $\mathfrak{t}$ is given by 
\begin{displaymath}
\mathfrak{t}=\left(\,\Yvcentermath1\young(1369,257,4)\,,\,\Yvcentermath1\young(8)\, \right),
\end{displaymath}
then
\begin{align*}
\lambda(\mathfrak{t})&=\left(\,\Yvcentermath1\young(\oneone\oneone\threeone\twotwo,\oneone\twoone\onetwo,\twoone)\,,\,\Yvcentermath1\young(\onetwo)\, \right) &&\mathtt{T}^{\lambda}=\left(\,\Yvcentermath1\young(\oneone\oneone\oneone,\twoone\twoone,\threeone)\,,\,\Yvcentermath1\young(\onetwo\onetwo,\twotwo)\, \right).
\end{align*}
\end{exm}    

When it comes to actually writing down examples of such tableaux, we will let $i_k$ represent the entry $(i,k)$, such a format being better suited to the context of a diagram.   

\paragraph{Homomorphisms}
Given a $\mu$-tableau $\mathfrak{t}$ we write $\lambda(\mathfrak{t})$ for the $\mu$-tableau of type $\lambda$ obtained by replacing every entry in $\mathfrak{t}$ with $(i,k)$ whenever that entry appears in the $i$-th row of the $k$-th component of $\mathfrak{t}^{\lambda}$. 

Let $\lambda$ and $\mu$ be multipartitions with $\lambda\unlhd\mu$ and let $\mathtt{T}$ be a semistandard $\mu$ tableau of type $\lambda$.  
We can define a homomorphisms from $M^{\lambda}$ to $S^{\mu}$ by 
\begin{displaymath}
\varphi_{\mathtt{T}}\left(m_{\lambda}h\right)=\check{\mathscr{H}}^{\mu}+\left(m_{\mu}\sum _{\stackrel{\mathfrak{t}\in\mathrm{Std}\left(\mu\right)}{\lambda(\mathfrak{t})=\mathtt{T}}}T_{d(\mathfrak{t})}\right)h
\end{displaymath} 
for $h\in\mathscr{H}$.
\begin{exm}
Suppose that $\lambda=((3,2,1),(2,2))$ , $\mu=((4,2,1),(2,1)),$ and that $\mathtt{T}\in\mathcal{T}_0(\mu,\lambda)$ is given by
\begin{displaymath}
\mathtt{T}=\left(\,\Yvcentermath1\young(\oneone\oneone\oneone\twoone,\twoone\threeone,\onetwo)\,,\,\Yvcentermath1\young(\onetwo\twotwo,\twotwo)\,\right).
\end{displaymath}
Then 
\begin{align*}
\varphi_{\mathtt{T}}(m_{\lambda}h)=\check{\mathscr{H}}^{\mu}&+(1+T_1)(1+T_2+T_2T_1)(1+T_3+T_3T_2+T_3T_2T_1)\\ &\times(1+T_5)(1+T_8)\\ 
&\times(L_1-Q_2)(L_2-Q_2)\cdots(L_7-Q_2)\\
&\times(1+T_4)(1+T_7)(1+T_9)
h.
\end{align*}
\end{exm}

\section{Main Results}
Given some multipartition $\lambda$, we can define the natural projection $\pi_{\lambda}:M^{\lambda}\rightarrow S^{\lambda}$ by $$\pi_{\lambda}(m_{\lambda})=\check{\mathscr{H}}^{\lambda}+m_{\lambda}.$$
Letting $\Theta:M^{\lambda}\rightarrow S^{\mu}$ be any homomorphism, we see that $\Theta$ factors through $S^{\lambda}$ if and only if $\Theta(m_{\lambda}h)=0 $ for all $h\in\mathscr{H}$ with $m_{\lambda}h\in\check{\mathscr{H}}^{\lambda}$.  If this is the case, then $\Theta$ determines a homomorphism $\tilde{\Theta}:S^{\lambda}\rightarrow S^{\mu}$  such that the following diagram commutes:
\begin{displaymath}
\xymatrix{
M^{\lambda}\ar[rr]^{\pi_{\lambda}}\ar[dr]_{\Theta} &&S^{\lambda}\ar[ld]^{\tilde{\Theta}}\\ &S^{\mu}
}
\end{displaymath} 

Our interest lies in providing some criteria for such homomorphisms to satisfy this condition. 

The strategy pursued in this paper, like the preceding discussion, closely follows that of \cite{Lyle1}.  In particular, we construct two families of elements $\mathfrak{d}^{(s)}_{d,t}$ and $\mathfrak{l}^{(s')}$ with the property that $\Theta(m_\lambda h)=0$ for all $h\in\mathscr{H}$ such that $m_\lambda h\in\check{\mathscr{H}}^\lambda$ if and only if $\Theta(m_{\lambda}\mathfrak{d}^{(s)}_{d,t})=0$ and $\Theta(m_{\lambda}\mathfrak{l}^{(s')})=0$ for all indices over which such elements are defined.  
This amounts to proving that the right ideal generated by the set of elements of the form $m_{\lambda}\mathfrak{d}^{(s)}_{d,t}$ and $m_{\lambda}\mathfrak{l}^{(s')}$ is equal to the right ideal $M^{\lambda}\cap\check{\mathscr{H}}^{\lambda}$.     

\subsection{Semistandard Homomorphisms}

Let $\nu$ be any $r$-multicomposition of $n,$ and for $i=1,2,\ldots, r$ let $\rho_i(\nu)$ stand for the number of rows appearing in the $i$-th component of $\nu$.  For integers $1\leq s\leq r$ and $0\leq k\leq \rho_s(\nu)$, define \begin{displaymath}
\overline{\nu}_{(k,s)}=\sum_{j=1}^{s}\left|\nu^{(j-1)}\right|+\sum_{i=0}^k \nu^{(s)}_i\phantom{iii}\text{ and }\phantom{iii}\overline{\nu}_{(s)}=\overline{\nu}_{(0,s)}=\sum_{j=1}^s \left|\nu^{(j-1)} \right|.
\end{displaymath}
Also, if $\eta$ is any composition let $\overline{\eta}=|\eta|$ and let $\mathfrak{S}_{(m,\eta)}$ denote the symmetric group acting on the letters $m+1,m+2,\ldots,m+\overline{\eta}$ for some positive integer $m$. With $\mathscr{D}_{m,\eta}$ denoting the set of right coset representatives of $\mathfrak{S}_{(m+1, \eta_1,\ldots,\eta_l)}\cap\mathfrak{S}_{( m,\eta)}$ in $\mathfrak{S}_{(m,\eta)}$ of minimal length, we set 
\begin{displaymath}C\left( m\,;\,\eta\right)=C\left(m\,;\,(\eta_1,\eta_2,\ldots,\eta_l)\right)=\sum_{w\in\mathscr{D}_{m,\eta}}T_w.\end{displaymath}

Suppose that $\mathfrak{t}$ is the $\eta$ tableau in which the entries $m+1,m+2,\ldots, m+\overline{\eta}$ appear in numerical order along the rows, starting with the first row and working down to the last.  Then $\mathrm{C}(m; \eta)$ is simply the sum of elements of $\mathscr{H}$ which are indexed by permutations in $\mathfrak{S}_{(m,\eta)}$ which keeps $\mathfrak{t}$ row standard.  

Then, for $1\leq s\leq r, 1\leq d<\rho_s(\nu)$, and $1\leq t\leq \nu^{(s)}_{d+1}$, define 
$$\mathfrak{d}^{(s)}_{d,t}=\mathrm{C}\left(\overline{\nu}_{(d-1,s)}\,;\,\left(\nu^{(s)}_d,t\right)\right).$$
For the sake of convenience, in what follows we will let $\mathrm{def}(\nu, \mathfrak{d})$ denote the set of triples $(s,d,t)$ for which $\mathfrak{d}^{(s)}_{d,t}$ is defined.

These elements fulfill a similar role to that played by the $h_{d,t}$ elements of {\cite{Lyle1}, and hence also that of the $\psi_{d,t}$ maps involved in the Kernel Intersection Theorem of \cite{DJ1}.  

Working with multipartitions necessitates we introduce a second family of elements of $\mathscr{H}$ if we're to generate the ideal $M^\nu\cap\check{\mathscr{H}}^\nu$.   Let $\mathrm{def}(\nu,\mathfrak{l})$ denote the set 
  \begin{displaymath}
  \mathrm{def}(\nu,\mathfrak{l})=\{s'\in\mathbb{Z}:1\leq s'\leq r-1, \nu^{(s'+1)}\neq\emptyset\}.
  \end{displaymath} Then, for every $s'\in\mathrm{def}(\nu,\mathfrak{l})$ we define elements
\begin{displaymath}
\mathfrak{l}^{(s')}=\left(L_{\overline{\nu}_{(s'+1)}+1}-Q_{s'+1}\right).
\end{displaymath}

\begin{exm}
Let $\lambda=((3,1),(2,2),(2,1,1))$.  Then 
\begin{align*}
&\mathfrak{d}^{(1)}_{1,1}=1+T_3+T_3T_2+T_3T_2T_1 &&\mathfrak{d}^{(2)}_{1,1}=1+T_6+T_6T_5\\
&\mathfrak{d}^{(2)}_{1,2}=1+T_6+T_6T_5+T_6T_7+T_6T_7T_5+T_6T_7T_5T_6 
\\ &\mathfrak{d}^{(3)}_{1,1}=1+T_{10}+T_{10}T_9 &&\mathfrak{d}^{(3)}_{2,1}=1+T_{11}\\
&\mathfrak{l}^{(1)}=L_5-Q_2 &&\mathfrak{l}^{(2)}=L_9-Q_3.
\end{align*}
\end{exm}

Setting$$\mathbf{L}:=\{m_{\nu}\mathfrak{l}^{(s')}:s'\in\mathrm{def}(\nu,\mathfrak{l})\} \text{ and }\mathbf{D}:=\{m_{\nu}\mathfrak{d}^{(s)}_{d,t}:(s,d,t)\in\mathrm{def}(\nu,\mathfrak{d})\},$$
we then state the main results of this paper.  
\begin{thm}\label{Ideal1}
Let $\lambda$ be an $r$-multipartition of $n$ and let $\mathfrak{I}$ be the right ideal generated by the set $\mathbf{L}\cup\mathbf{D}$.  Then
\begin{displaymath}\mathfrak{I}=M^{\lambda}\cap\check{\mathscr{H}}^{\lambda}.\end{displaymath}
\end{thm}

The proof of the above theorem is technical and fairly lengthy, and so we postpone it until the next section of the paper.  For now we focus on the consequences and possible applications of Theorem \ref{Ideal1}.  
\begin{cor}\label{Build}
Suppose that $\lambda$ and $\mu$ are multipartitions with $\lambda\unlhd\mu$, and that $\Theta:M^{\lambda}\rightarrow S^{\mu}$.  Then, for each $h\in\mathscr{H}$ with $m_{\lambda}h\in\check{\mathscr{H}}^{\lambda}$, 
\begin{displaymath}
\Theta(m_{\lambda}h)=0 \text{ if and only if }\Theta \left(m_{\lambda}\mathfrak{d}^{(s)}_{d,t}\right)=0\text{ and } \Theta\left(m_{\lambda}\mathfrak{l}^{(s')}\right)=0
\end{displaymath}
for all $(s,d,t)\in\mathrm{def}(\lambda, \mathfrak{d})$, and all $s'\in\mathrm{def}(\lambda,\mathfrak{l})$.
\end{cor}  

\paragraph{Constructing Homomorphisms} When it comes to explicitly constructing homomorphisms between Specht modules our task is to determine precisely when the conditions laid out in Corollary \ref{Build} hold. In principle this can be done by examining the action of the elements of $\mathbf{L}$ and $\mathbf{D}$ on the basis elements of $\mathrm{Hom}(M^\lambda,S^\mu)$.  

We now describe $r$-multicompositions of $n$ associated with the elements just defined. Suppose that $\lambda$ is a multipartition and that $(s,d,t)\in\mathrm{def}(\lambda, \mathfrak{d})$ and $s'\in\mathrm{def}(\lambda,\mathfrak{l} )$.  We can define  multicompositions $\lambda\cdot\mathfrak{d}^{(s)}_{d,t}$ and $\lambda\cdot\mathfrak{l}^{(s')}$  given by
\begin{displaymath}
\left(\lambda\cdot\mathfrak{d}^{(s)}_{d,t}\right)^{(l)}_i=\left\{
\begin{array}{ll}
\lambda^{(s)}_d+t &\text{if }l=s\text{ and }i=d\\
\lambda^{(s)}_{d+1}-t &\text{if }l=s \text{ and }i=d+1\\
\lambda^{(l)}_i &\text{otherwise}.
\end{array}
\right.
\end{displaymath} 
and
\begin{displaymath}
\left(\lambda\cdot\mathfrak{l}^{(s')}\right)^{(k)}_j=\left\{
\begin{array}{ll}
1 &\text{if }k=s'\text{ and }j=\rho_{s'}(\lambda)+1\\
\lambda^{(s'+1)}_1-1&\text{if }k=s'+1\text{ and }j=1\\
\lambda^{(k)}_j&\text{otherwise.}
\end{array}
\right.
\end{displaymath}
In other words, $\lambda\cdot\mathfrak{d}^{(s)}_{d,t}$ is the multicomposition whose diagram is formed by raising the last $t$ nodes of the $(d+1)$-th row of the $s$-th component of $[\lambda]$ to the end of the $d$-th row of the same component.  Similarly $\lambda\cdot\mathfrak{l}^{(s')}$ is the multicomposition whose diagram is obtained by removing the last node from the first row of component $s'+1$ of $[\lambda]$ and inserting a new row consisting of a single node at the bottom of component $s'$.
\begin{exm}
Let $\lambda=((3,1),(2,2),(2,1,1))$.  Then
\begin{align*}
&\left[\lambda\cdot\mathfrak{d}^{(2)}_{1,2}\right]=\left(\, \Yvcentermath1\yng(3,1)\,,\,\Yvcentermath1\yng(4)\,,\,\Yvcentermath1\yng(2,1,1)\, \right), \text{and}
\\ &\left[\lambda\cdot\mathfrak{l}^{(1)}\right]=\left(\,\Yvcentermath1\yng(3,1,1)\,,\,\Yvcentermath1\yng(1,2)\,,\,\Yvcentermath1\yng(2,1,1)\, \right).
\end{align*}
\end{exm}  

The idea behind defining these various multicompositions is that if we have a homomorphism $\Theta:M^\lambda\rightarrow S^\mu$ which can be expressed as a linear combination
\begin{displaymath}
\Theta=\sum_{\mathtt{T}\in\mathcal{T}_0(\mu,\lambda)} a_\mathtt{T}\varphi_{\mathtt{T}}
\end{displaymath}
with $a_\mathtt{T}\in\mathbb{F}$, then it should be possible to show that acting on each $\varphi_\mathtt{T}$ by, say, $\mathfrak{d}^{(s)}_{d,t}\in\mathbf{D}$ produces a linear combination of homomorphisms
\begin{displaymath}
\psi_\mathtt{S}:M^{\lambda\cdot\mathfrak{d}^{(s)}_{d,t}}\rightarrow S^\mu
\end{displaymath}
where $\mathtt{S}$ ranges over the semistandard $\mu$-tableaux of type $\lambda\cdot\mathfrak{d}^{(s)}_{d,t}$. 

Given that the tableaux $\mathtt{S}$ ranges over are semistandard, the set of homomorphisms determined by them are linearly independent \cite[Corollary 6.14]{DJM}. 
The condition in Corollary \ref{Build} then becomes a matter of collecting terms in $\Theta\mathfrak{d}^{(s)}_{d,t}$ and setting the coefficients in $\mathbb{F}$ to zero.  A similar procedure applies to when we're acting on $\Theta$ by an element of $\mathbf{L}$.  We provide a demonstration of this process in the final section of this paper.  

\section{Proof of Theorem \ref{Ideal1}}
\paragraph{Preliminaries}Our proof centres  on the ability to confine a lot of our attention to the subalgebra of $\mathscr{H}$  isomorphic to the Iwahori-Hecke algebra of type $\mathrm{A}$, thus allowing us to use results from that setting.

A major aspect of our approach will be to treat an $r$-multipartition of $n$ as a composition of $n$.   In particular, let $\lambda$ be  a multipartition of $n$ in $r$ parts and let $\alpha(\lambda)$ be the composition of $n$ determined by the diagram formed by `stacking' the components of $[\lambda]$ on top of one another, such that $[\lambda^{(i+1)}]$ appears immediately below $[\lambda^{(i)}]$ for all $1\leq i\leq r-1$.   Note also that a standard $\lambda$-tableau in $\mathscr{H}$ corresponds to a row-standard $\alpha(\lambda)$-tableau.

For any pair of multipartitions $\lambda$ and $\mu$ we can perform the same kind of procedure on $\mu$-tableaux of type $\lambda$.  If $\mathtt{S}\in\mathcal{T}(\mu,\lambda)$, then define $\alpha(\mathtt{S})$ to be the $\alpha(\mu)$-tableau of type $\alpha(\lambda)$ formed by stacking the components of $\mathtt{S}$ in the same manner as was done in the previous construction.  To complete this process, relabel every entry in $\alpha(\mathtt{S})$ using the rule  
\begin{displaymath}(u.v)\mapsto u+\rho_{v-1}(\lambda).
\end{displaymath}
It's clear that if $\mathtt{S}$ is semistandard, then $\alpha(\mathtt{S})$ is row-semistandard.  
\begin{exm}
If $\mu=((3,2),(2,1,1))$, then $\alpha({\mu})=(3,2,2,1,1)$.   Additionally, if $\lambda=((2,1),(3, 2, 1))$ and $\mathtt{S}\in\mathcal{T}_0(\mu,\lambda)$ is given by
\begin{displaymath}
\mathtt{S}=\left(
\Yvcentermath1\young(\oneone\oneone\twotwo,\twoone\onetwo)\,,\,
\Yvcentermath1\young(\onetwo\onetwo,\twotwo,\threetwo) 
\right)
\end{displaymath}
then
\begin{displaymath}
\alpha(\mathtt{S})=\Yvcentermath1\young(114,23,33,4,5)
\end{displaymath}
\end{exm}

The first stage of proving the statement of Theorem \ref{Ideal1} is to subdivide the task into proving that $\mathfrak{I}\subseteq M^\lambda\cap\check{\mathscr{H}}^\lambda$ and then that $\mathfrak{I}\supseteq M^\lambda\cap\check{\mathscr{H}}^\lambda$. The first of these follows almost immediately from the definitions.    

\begin{lem}\label{firsty}
For every $r$-multipartition $\lambda$ of $n$ 
\begin{displaymath}\mathfrak{I}\subseteq M^{\lambda}\cap\check{\mathscr{H}}^{\lambda}.\end{displaymath}
\end{lem}

\begin{proof}
Consider $m_\lambda\mathfrak{d}^{(s)}_{d,t}$ for some $(d,t,s)\in\mathrm{def}(\mathfrak{d})$ and observe that $\mathfrak{d}^{(s)}_{d,t}$ and $u^+_\lambda$ commute with one another, since any particular element of $\mathbf{D}$ permutes only entries within a given component. We then have 
\begin{displaymath}
m_\lambda\mathfrak{d}^{(s)}_{d,t}=\left(x_\lambda\mathfrak{d}^{(s)}_{d,t}\right)u^+_\lambda.
\end{displaymath}

Let $\nu=\lambda\cdot\mathfrak{d}^{(s)}_{d,t}$ and let $\mathtt{A}$ be the row semistandard $\lambda$-tableau of type $\nu$ derived from $\mathtt{T}^\lambda$ by changing the first $t$ entries in row $d+1$ of component $s$ from $(d+1,s)$ to $(d,s)$. Then
\begin{displaymath}
u^+_\lambda=u^+_\nu\phantom{iiii}\text{and}\phantom{iiii}x_\lambda\mathfrak{d}^{(s)}_{d,t}=x_\lambda\sum_{\stackrel{x\in\mathscr{D}_\lambda}{\nu(\mathfrak{t}^{\lambda }x)}} T_x
\end{displaymath} 
which, viewing $\lambda$ and $\nu$ as multicompositions, in turn implies that 
\begin{displaymath}
m_\lambda\mathfrak{d}^{(s)}_{d,t}=\left(x_\lambda\sum_{\stackrel{x\in\mathscr{D}_\lambda}{\nu(\mathfrak{t}^\lambda x)=\mathtt{A}}}T_x\right)u^+_\lambda =hx_\nu u^+_\nu=h m_\nu
\end{displaymath}
for some $h\in\mathscr{H}$.  Note that the second equality comes from \cite[4.6]{Mathas1} and our ability to regard $\lambda$ as a composition. Given that by definition $\nu$ dominates $\lambda$, we then have $m_\lambda\mathfrak{d}^{(s)}_{d,t}\in M^\lambda\cap \check{\mathscr{H}}^\lambda$.

Consider now $m_\lambda \mathfrak{l}^{(s')}$ for $s'\in\mathrm{def}(\lambda,\mathfrak{l})$; we proceed in a similar manner as in the previous case. If we let $\nu=\lambda\cdot\mathfrak{l}^{(s')}$ and let $\mathtt{A}'$ be the row semistandard $\lambda$ tableau of type $\nu$ derived from $\mathtt{T}^\lambda$ by changing the entry in the first node of the first row of component $s'+1$ from $(1,s'+1)$ to $(\rho_{s'}(\lambda)+1,s')$, then 
\begin{displaymath}
x_\lambda\sum_{\stackrel{x\in\mathscr{D}_\lambda}{\nu(\mathfrak{t}^\lambda x)=\mathtt{A}'}}T_x=h'x_\nu
\end{displaymath}
for some $h'\in\mathscr{H}$. Since $u^+_\lambda\mathfrak{l}^{(s')}=u^+_\nu$ and $\lambda\lhd\nu$ we then have $m_\lambda\mathfrak{l}^{(s')}\in M^\lambda\cap\check{\mathscr{H}}^\lambda$ as required.  
\end{proof}

Showing that the inclusion holds in the opposite direction is a far more involved process.  The focus of this part of the proof concerns the following elements of $\mathscr{H}$: If $\nu$ and $\lambda$ are multipartitions of $n$, and if $\mathtt{S}\in\mathcal{T}_0(\nu,\lambda)$ and $\mathfrak{t}\in\mathrm{Std}(\nu)$, the element $m_{\mathtt{S}\mathfrak{t}}$ of $\mathscr{H}$ is defined by 
\begin{displaymath}
m_{\mathtt{S}\mathfrak{t}}=\sum_{\stackrel{\mathfrak{s}\in\mathrm{Std}(\nu)}{\lambda(\mathfrak{s})=\mathtt{S}}}m_{\mathfrak{s}\mathfrak{t}}.
\end{displaymath}

The significance of these elements is that if we take $\nu$ to range over the multipartitions of $n$, they form an $\mathbb{F}$-basis of $M^\lambda$. We then have

\begin{lem}[{\cite[Theorem 4.5]{Mathas2}}]\label{secondy}
The right ideal $M^{\lambda}\cap\check{\mathscr{H}}^{\lambda}$ has a basis
\begin{displaymath}
\left\{m_{\mathtt{S}\mathfrak{t}}:\mathtt{S}\in\mathcal{T}_0(\nu,\lambda),\mathfrak{t}\in\mathrm{Std}(\nu)\text{ for }\nu\vdash n\text{ with }\lambda\lhd\nu\right\},
\end{displaymath}
\end{lem}
Our proof now becomes a matter of showing that each element of this basis can be expressed in the form 
\begin{equation}\label{point}
m_{\mathtt{S}\mathfrak{t}^\nu}=\sum_{(d,t,s)\in\mathrm{def}(\lambda, \mathfrak{d})}\gamma ^{(s)}_{d,t}m_\lambda\mathfrak{d}^{(s)}_{d,t}h^{(s)}_{d,t}+\sum_{s'\in\mathrm{def}(\lambda, \mathfrak{l})}\gamma^{(s')}m_\lambda\mathfrak{l}^{(s')}h^{(s')}
\end{equation}
for some $\gamma^{(s)}_{d,t},\gamma^{(s')}\in\mathbb{F}$ and $h^{(s)}_{d,t},h^{(s')}\in\mathscr{H}$.  

 In order to do this, we seek to better describe the various appropriate $m_{\mathtt{S}\mathfrak{t}}$ elements. Specifically, we want a method of expressing them in the form of $m_\lambda h$ for elements $h$ of $\mathscr{H}$. 
As a first step towards this goal we provide the following lemma.
\begin{lem}\label{Commutation}
Suppose that $w\in\mathfrak{S}_n$ is such that it stabilizes the components of $\mathfrak{t}^\nu$.  Then $T_wu^+_\nu=u^+_\nu T_w$.
\end{lem}
\begin{proof} If $w$ stabilizes the components of $\mathfrak{t}^\nu$, then $w$ can be expressed as a series of disjoint cycles $w=w_1w_2\cdots w_r$ where, if $\mathfrak{S}_{\nu^{(i)}}$ is the symmetric group acting on
\begin{displaymath}
\overline{\nu}_{(i)}+1,\overline{\nu}_{(i)}+2,\ldots,\overline{\nu}_{(i+1)},
\end{displaymath}
then $w_i\in\mathfrak{S}_{\nu^{(i)}}$.  

Hence each $w_i$ can be expressed in terms of generators 
\begin{displaymath}s_{\overline{\nu}_{(i)}+1},s_{\overline{\nu}_{(i)}+2}, \ldots, s_{\overline{\nu}_{(i+1)-1}}.\end{displaymath}
which means that $w$ does not involve $s_{\overline{\nu}_{(i)}}$ for any $1\leq i\leq r$.  Hence $T_w$ maybe expressed as $T_{w_1}T_{w_2}\cdots T_{w_r}$, each term commuting with $u^+_\nu$ by Proposition \ref{Jucys} and the definition of $u^+_\nu$.  
\end{proof}

 Next, we  introduce a little more notation, mimicking that appearing in \cite{Lyle2}, which will prove useful both in our next result and throughout this paper:

\begin{itemize}\item  Let $\nu$ and $\lambda$ be compositions and let  $\mathtt{S}$  be a $\nu$-tableau of type $\lambda$.  If $(i,j)$ is a pair of integers with $1\leq j \leq r$ and $1\leq i\leq \rho_j(\lambda)$, and $(k,l)$ is a pair of integers with $1\leq l \leq r$ and $1\leq k\leq \rho_l(\nu)$, then we define $\mathtt{S}^{(i,j)}_{(k,l)}$ to be the number of entries appearing in the $k$-th row of the $l$-th component of $\mathtt{S}$ which are equal to $(i,j)$.
\item With the previous item of notation in place, let $\Gamma_{(x,y)}$   be the sequence
\begin{align*}
&\mathtt{S}_{(x,y)}^{(x.y)},\mathtt{S}_{(x,y)}^{(x+1,y)},\ldots,\mathtt{S}^{(\rho_y(\lambda),y)}_{(x,y)},\\
&\mathtt{S}^{(1,y+1)}_{(x,y)},\mathtt{S}^{(2,y+1)}_{(x,y)},\ldots,\mathtt{S}^{(\rho_{y+1}(\lambda),y+1)}_{(x,y)},\\
&\vdots\\
&\mathtt{S}^{(1,r)}_{(x,y)},\mathtt{S}^{(2,r)}_{(x,y)},\ldots,\mathtt{S}^{(\rho_r(\lambda),r)}_{(x,y)}.
\end{align*} 
\item  Let $\mathtt{S}$ be a $\nu$-tableau of type $\lambda$ and let $\mathfrak{t}_{\mathtt{S}}$ be the row-standard $\lambda$-tableau in which $i$ occupies a node in row $u$ of component $v$ if the place occupied by $i$ in $\mathfrak{t}^{\nu}$ is occupied by $(u,v)$ in $\mathtt{S}$.  If $\mathfrak{t}_{\mathtt{S}}=\mathfrak{t}^{\lambda}\cdot w$, we set $T_{\mathtt{S}}=T_w$. 
\end{itemize} 

\begin{lem}\label{Com1}
Suppose that $\lambda$ and $\nu$ are multipartitions of $n$ and that $\mathtt{S}\in\mathcal{T}_0(\nu,\lambda)$.  Then
 \begin{displaymath}
 m_{\mathtt{S}\mathfrak{t}^{\nu}}=x_{\lambda}T_{\mathtt{S}}u^+_{\nu}\prod_{x,y\geq 1}\mathrm{C}\left(\overline{\nu}_{(x-1,y)}:\Gamma_{(x,y)} \right)
 \end{displaymath}
\end{lem}
\begin{proof}By treating $x_\nu$ as the row stabilizer of the composition $\alpha({\nu})$ and $\mathtt{S}$ as a row-semistandard $\alpha(\nu)$ tableau of type $\alpha(\lambda)$ we can apply \cite[Corollary 3.7]{Lyle2} to yield  
\begin{align*}
m_{\mathtt{S}\mathfrak{t}^{\nu}}&=\sum_{\stackrel{x\in\mathscr{D}_\nu}{\lambda(\mathfrak{t}^\nu x)=\mathtt{S}}}T^*_xx_\nu u^+_\nu \\
&=x_\lambda T_{\mathtt{S}}\prod_{x,y\geq1} \mathrm{C}\left(\overline{\nu}_{(x-1,y)}:\Gamma_{(x,y)}\right) u^+_\nu
\end{align*}
The result then follows from the observation that the permutations indexing the terms of each $\mathrm{C}$-element permute only the entries in a given row of $\mathfrak{t}^\nu$, and hence, by Lemma \ref{Commutation}, the product of such elements commutes with $u^+_\nu$.
\end{proof}

\begin{exm}
Let $\lambda=((3,2,2), (3,1,1))$ and $\nu=((4,3,2,1),(2))$ and let $\mathtt{S}\in\mathcal{T}_0(\nu,\lambda)$ be given by 
\begin{displaymath}
\mathtt{S}=\left(\,\Yvcentermath1\young(\oneone\oneone\oneone\onetwo,\twoone\twoone\threeone,\threeone\threetwo,\twotwo)\,,\,\Yvcentermath1\young(\onetwo\onetwo)\,  \right)
\end{displaymath}
Then 
\begin{align*}m_{\mathtt{S}\mathfrak{t}^{\nu}}&= (1+T_8+T_{8,9})(1+T_6)T_{7,6,5,4}T_{11,10,9}T_{11,10}m_\nu \\&=x_\lambda T_{7,6,5,4}T_{11,10,9}T_{11,10} u^+_{\nu}(1+T_3+T_{3,2}+T_{3,2,1})(1+T_6+T_{6,5})\\
& \times (1+T_8).
\end{align*}
\end{exm}

The consequence of this lemma is that our proof that \eqref{point} holds now becomes a matter of investigating how $T_{\mathtt{S}}$ interacts with $u^+_{\nu}$, for every multipartition $\nu$ of $n$ dominating $\lambda$ and every $\mathtt{S}\in\mathcal{T}_0(\lambda,\nu)$.  This breaks down into two cases: when all  of the components of $\lambda$ are the same size as their counterparts in $\nu$, and when at least one component of $\nu$ is larger than the corresponding component of $\lambda$.  

Of these cases, the former is considerably easier and it turns out that we can largely ignore $u^+_\nu$. The second case is far more complicated and involves showing that the act of trying to commute $T_{\mathtt{S}}$ `as far as we can' through $u^+_\nu$ results in an element of the desired form.  

To distinguish these two cases, we  will informally view $\nu$ as being constructed from $\lambda$ by a process moving  nodes around the diagram $[\lambda]$.  In the former case  $[\nu] $ can be formed by moving nodes within the components of $[\lambda]$, a situation we will describe as $\nu$ being a componentwise shift of $\lambda$.  We refer to the latter case as $\nu$ as being a cross component shift of $\lambda$. 

\paragraph{Componentwise Shifts.}
For every integer $s$ with $1\leq s \leq r$, let $\lambda$ and $\nu$ be multicompositions such that $|\lambda^{(s)}|=|\nu^{(s)}|$ and $\lambda\lhd\nu$.  We may then view the $s$-th component of both multipartitions as partitions of the same size, which we denote by $n_s$. Additionally, since $\overline{\nu}_{(s)}=\overline{\lambda}_{(s)}$ and $\lambda\unlhd \nu$ we have that $\lambda^{(s)}\unlhd\nu^{(s)}$ for every value of $s$. 

What's more is that the set of entries appearing in a given component of $\mathfrak{t}^{\lambda}$ is the same  as that which appear in the same component of $\mathfrak{t}^{\nu}$.  Namely, the set of entries appearing in the $s$-th component of both tableaux is
\begin{displaymath}
\left\{\overline{\lambda}_{(s)}+1,\overline{\lambda}_{(s)}+2,\ldots,\overline{\lambda}_{(s+1)}\right\}.
\end{displaymath}

Let $\mathtt{S}\in\mathcal{T}_0(\nu,\lambda)$. Since $\mathtt{S}$ is semistandard, the permutation $w\in\mathfrak{S}_n$ associated with $T_\mathtt{S}$ permutes the entries of the initial $\lambda$-tableau component-wise; that is, the set of entries in a given component of $\mathfrak{t}^\lambda$ is the same as the set of entries in the same component of $\mathfrak{t}^{\lambda}\cdot w$.  Otherwise, we would have an entry $(u,v)$ appearing in some component $v'$ of $\mathtt{S}$ where $v<v'$ contradicting the assumption that $\mathtt{S}$ is a semistandard $\nu$ -tableau.

Let $w_s$ be the permutation in $\mathfrak{S}_X$, where $X=\{\overline{\lambda}_{(s)}+1,\overline{\lambda}_{(s)}+2,\ldots,\overline{\lambda}_{(s+1)}\}$, for which the $s$-th component of $\mathfrak{t}^\lambda\cdot w$ is the same as $\mathfrak{t}^{\lambda}\cdot w_s$, and let $T_{\mathtt{S}^{(s)}}=T_{w_s}$.  Using this notation, we have the following decomposition which goes some way to showing how $m_{\mathtt{S}\mathfrak{t}^\nu}$ can be expressed in the form $m_\lambda h$, where $h$ is a product of terms, each of which acts only on a particular component of $\lambda$.     

\begin{lem}
Let $\lambda$ and $\nu$ be multipartitions such that $\lambda \unlhd \nu$ and $|\lambda^{(i)}|=|\nu^{(i)}|$ for all $1\leq i\leq r$.  Then, whenever $\mathtt{S}\in\mathcal{T}_0(\nu,\lambda)$,  
\begin{displaymath}
T_\mathtt{S}=\prod_{s=1}^r T_{\mathtt{S}^{(s)}}
\end{displaymath} 
and
\begin{align*}
&T_{\mathtt{S}}\prod_{x,y\geq 1}\mathrm{C}\left(\overline{\nu}_{(x-1,y)}:\Gamma_{(x,y)}\right)
=\prod_{y=1}^r \left(T_{\mathtt{S}^{(y)}}\prod_{x\geq1} \mathrm{C}\left(\overline{\nu}_{(x-1,y)}:\Gamma_{(x,y)}\right)\right).
\end{align*}
\end{lem}

\begin{proof}
It's clear that $w=w_1w_2\cdots w_r$ as a product of disjoint cycles, since for every value of $s$ the permutation fixes all entries in $\mathfrak{t}^\lambda$ bar those appearing in component $s$.   Then
\begin{displaymath}
T_{\mathtt{S}}=\prod_{s=1}^r T_{\mathtt{S}^{(s)}}
\end{displaymath}
By similar reasoning $T_{\mathtt{S}^{(s)}}$ commutes with 
\begin{displaymath}
\mathrm{C}\left(\overline{\nu}_{(x-1,s')}:\Gamma_{(x,y)}\right)
\end{displaymath} whenever $s\neq s'$.  This then delivers the second part of the statement.  
\end{proof}

\begin{lem}\label{Thirdy}
Let $\lambda$ and $\mu$ be multipartitions with $\lambda\lhd\nu$ and $|\lambda^{(i)}|=|\nu^{(i)}|$ for all $1\leq i\leq r$, and let $\mathtt{S}\in\mathcal{T}_0(\nu,\lambda)$.  Then $m_{\mathtt{S}\mathfrak{t}^{\nu}}\in\mathfrak{I}$.
\end{lem}

\begin{proof}
If $|\lambda^{(i)}|=|\nu^{(i)}|$ for all $1\leq i\leq r$, then $u^+_{\lambda}=u^+_{\nu}$.   By Lemma \ref{Com1} , we  have
  \begin{align*}
  m_{\mathtt{S}\mathfrak{t}^{\nu}}=\sum_{\stackrel{
  \mathfrak{s}\in\mathrm{Std}(\nu)}{
  \lambda(\mathfrak{s})=\mathtt{S}}}m_{\mathfrak{s}\mathfrak{t}^{\nu}}&=\left(x_{\lambda}T_{\mathtt{S}}u^+_\nu\right)\prod_{x,y\geq 1}\mathrm{C}\left(\overline{\nu}_{(x-1,s)}:\Gamma_{(x,y)} \right)
  \end{align*}
  
  We now claim that $T_\mathtt{S}$ and $u^+_\nu$ commute, but this is clear since the permutation associated with $\mathtt{S}$ only permutes entries within the individual components and not across them.   Hence
\begin{displaymath}
m_{\mathtt{S}\mathfrak{t}^\nu}=u^+_{\lambda}x_{\lambda}T_{\mathtt{S}}\prod_{x,y\geq 1}\mathrm{C}\left(\overline{\nu}^{(y)}_{(x-1,y)}:\Gamma_{(x,y)} \right).
\end{displaymath}

Suppose now that $k$ is the least value of $s$ for which $\lambda^{(k)}\lhd \nu^{(k)}$.  Then it can be easily verified that 

\begin{equation}\label{Com2}
m_{\mathtt{S}\mathfrak{t}^{\nu}}=u^+_\lambda x_\lambda\prod_{s=k}^r \left(T_{\mathtt{S}^{(s)}}\prod_{i=1}^{\rho_s(\nu)}\mathrm{C}\left( \overline{\nu}_{(i-1,s)}:\Gamma_{(i,s)}\right) \right).
\end{equation}

 Now let $\mathfrak{S}_{\lambda^{(i)}}$ be the Young subgroup corresponding to the $i$-th component of $\lambda$ and define $$x_{{\lambda}^{(i)}}=\sum_{w\in\mathfrak{S}_{\lambda^{(i)}}}T_w$$ We may then write 
\begin{equation}\label{Com3}
x_\lambda=\prod_{i=1}^r x_{\lambda^{(i)}}
\end{equation}

Moreover,  $x_{\lambda^{(i)}}$ and $x_{\lambda^{(j)}}$ commute with one another whenever $i\neq j$ and so we may rearrange the terms of \eqref{Com3} however we please.  In particular, for $1\leq k\leq r$, let $x_{\lambda\setminus\lambda^{(k)}}$ denote the product $(x_{\lambda^{(1)}}x_{\lambda^{(2)}}\cdots x_{\lambda^{(k-1)}}x_{\lambda^{(k+1)}}\cdots x_{\lambda^{(r)}})$.  We may write $x_{\lambda}=x_{\lambda\setminus\lambda^{(k)}}x_{\lambda^{(k)}}$ and substitute this in \eqref{Com2} to yield

\begin{equation}\label{juggle}
m_{\mathtt{S}\mathfrak{t}^{\nu}}=u^+_\lambda x_{\lambda\setminus\lambda^{(k)}}\left(x_{\lambda^{(k)}}T_{\mathtt{S}^{(k)}}\prod_{i=1}^{\rho_k(\nu)}\mathrm{C}\left(\overline{\nu}_{(i-1,k)}:\Gamma_{(i-1,k)}\right) \right)h
\end{equation}  
where $h$ consists of the remaining terms of the product appearing in \eqref{Com2}.  The significance of the expression in brackets is that we may apply the results of \cite{Lyle2} almost directly. In particular, since
\begin{displaymath}
x_{\lambda^{(k)}}T_{\mathtt{S}^{(k)}}\prod_{i=1}^{\rho_k(\nu)}\mathrm{C}\left(\overline{\nu}_{(i-1,k)}:\Gamma_{(i-1,k)}\right)=m_{\mathtt{S}^{(k)}\mathfrak{t}^{\nu'}} \end{displaymath}
where $\lambda^{(k)}\lhd\nu^{(k)}=\nu'$ and $\mathtt{S}^{(k)}$ is a semistandard $\nu'$-tableaux of type $\lambda^{(k)}$, \cite[Theorem 2]{Lyle2} tells us that 
\begin{displaymath}
x_{\lambda^{(k)}}T_{\mathtt{S}^{(k)}}\prod_{i=1}^{\rho_k(\nu)}\mathrm{C}\left(\overline{\nu}_{(i-1,k)}:\Gamma_{(i-1,k)}\right)\in M^{\lambda^{(k)}}\cap\check{\mathscr{H}}^{\lambda^{(k)}}.
\end{displaymath}

We may then write the left hand side as a linear combination of elements of $\mathscr{H}$ of the form 
\begin{displaymath}
T^*_{d(\mathfrak{s})}x_{\mu} T_{d(\mathfrak{t})}
\end{displaymath}
where $\mu$ ranges over the partitions of $|\lambda^{(k)}|$ which dominate $\lambda^{(k)}$ and $\mathfrak{s}$ and $\mathfrak{t}$ are standard $\mu$-tableaux with entries taken from the set $\{\overline{\lambda}_{(k)}+1,\ldots, \overline{\lambda}_{(k+1)}\}$.

 Since $d(\mathfrak{s})$ permutes only these entries, substituting these terms into \eqref{juggle} gives us that $m_{\mathtt{S}\mathfrak{t}^\nu}$  is a linear combination of the terms
\begin{displaymath}
T^*_{d(\mathfrak{s})}u^+_\lambda x_{\lambda\setminus\lambda^{(k)}}x_\mu T_{d(\mathfrak{t})}h=T^*_{d(\mathfrak{s})}m_{\mu'}T_{d(\mathfrak{t})}h
\end{displaymath}
where $\mu'=(\lambda^{(1)},\lambda^{(2)},\ldots,\lambda^{(k-1)},\mu,\lambda^{(k+1)},\lambda^{(k+2)},\ldots,\lambda^{(r)})$.  It's clear that $\mu'$ is a multipartition and dominates $\lambda$, and that $\mathfrak{t}^{\mu'}\cdot d(\mathfrak{s})$ and $\mathfrak{t}^{\mu'}\cdot d(\mathfrak{t})$ are both standard $\mu'$-tableaux.  Hence $m_{\mathtt{S}\mathfrak{t}^\nu}\in\check{\mathscr{H}}^\lambda$ as required. 
 \end{proof}

\paragraph{Cross Component Shifts}
As was the case for componentwise shifts, we wish to express
\begin{displaymath}
m_{\mathtt{S}\mathfrak{t}^\nu}=\left(x_\lambda T_\mathtt{S}\prod_{x,y\geq 1}\mathrm{C}\left(\overline{\nu}_{(x-1,y)}:\Gamma_{(x,y)} \right)\right)u^+_\nu
\end{displaymath}
as an element of $\mathfrak{I}$ whenever $\lambda\lhd\nu$ and $\mathtt{S}\in\mathcal{T}_0(\nu,\lambda)$.  In this case we're interested in the situation where $|\lambda^{(i)}|<|\nu^{(i)}|$ for at least one $1\leq i\leq r-1$.  As before,
\begin{displaymath}
\prod_{x,y\geq 1}\mathrm{C}\left(\overline{\nu}_{(x-1,y)}:\Gamma_{(x,y)} \right)u^+_\nu=u^+_\nu\prod_{x,y\geq 1}\mathrm{C}\left(\overline{\nu}_{(x-1,y)}:\Gamma_{(x,y)} \right).
\end{displaymath}
However, it's not true in general that $T_\mathtt{S}u^+_\nu=u^+_\nu T_\mathtt{S}$ in this case, and so we will concentrate on how $T_\mathtt{S}$ interacts with $u^+_\nu$. 
As will be seen, once every term of the form $\mathrm{C}\left(\overline{\nu}_{(x-1,y)}:\Gamma_{(x,y)} \right)$ has been commuted past $u^+_\nu$, as in Lemma \ref{Com1}, they play no further part in the following proofs.  As such, it will be sufficient to show that $x_\lambda T_\mathtt{S}u^+_\nu\in\mathfrak{I}$. 

In order to proceed, we must introduce a particular family of $\lambda$-tableaux. As before, let $\lambda$ and $\nu$ be multipartitions with $\lambda\lhd\nu$, and let $\mathtt{S}\in\mathcal{T}_0(\nu,\lambda)$.
The tableaux in question will be defined inductively as follows:\begin{enumerate}
\item Set $\mathfrak{t}_{\mathtt{S}(0)}=\mathfrak{t}_{\mathtt{S}}$;
\item For $1\leq i <r-1$ define $\mathfrak{t}_{\mathtt{S}(i)}$ to be the $\lambda$-tableau obtained from $\mathfrak{t}_{\mathtt{S}(i-1)}$  by setting the entries in the first $i$ components to be the same as in $\mathfrak{t}^{\lambda}$.  Note that the first $i$ components of $\mathfrak{t}_{\mathtt{S}}$ contain only entries from the set $\{1,2,\ldots,\overline{\nu}_{(i+1)} \}$ and so there are then $\tau_i=\sum_{j=1}^i(|\nu^{(i)}|-|\lambda^{(i)}|)$ entries from the first $i$ components of $\mathfrak{t}^{\nu}$ appearing in the remaining $r-i$ components of our tableau.  Labelling these $\{x_1,x_2,\ldots,x_{\tau_i}\}$, such that $x_i<x_j$ whenever $i<j$,  we replace these elements using the rule
\begin{displaymath}
x_k\mapsto \overline{\lambda}_{(i)}+k
\end{displaymath} 
\end{enumerate}

\begin{exm}
If $\lambda=((2,1),(2,2),(2,2,1))$ and $\nu=((3,2,1),(3,1),(2))$, and $\mathtt{S}\in\mathcal{T}_0(\nu,\lambda)$ is given by
\begin{displaymath}
\mathtt{S}=\left( 
\Yvcentermath1\young(\oneone\oneone\twothree,\twoone\twotwo,\onethree)\,,\,
\Yvcentermath1\young(\onetwo\onetwo\threethree,\twotwo)\,,\,
\Yvcentermath1\young(\onethree\twothree)\,
\right)
\end{displaymath}
Then
\begin{align*}
&\mathfrak{t}_{\mathtt{S}(0)}=\left(
\Yvcentermath1\young(12,4)\,,\,\Yvcentermath1\young(78,5\ten)\,,\,\Yvcentermath1\young(6\eleven,3\twelve,9)\,\right) &&\mathfrak{t}_{\mathtt{S}(1)}=\left(\,\Yvcentermath1\young(12,3)\,,\,\Yvcentermath1\young(78,5\ten)\,,\,\Yvcentermath1\young(6\eleven,4\twelve,9)\,
\right)\\
&\mathfrak{t}_{\mathtt{S}(2)}=\left(\,\Yvcentermath1\young(12,3)\,,\,\Yvcentermath1\young(45,67)\,,\,\Yvcentermath1\young(9\eleven,8\twelve,\ten)\,\right).
\end{align*}
\end{exm}

With this definition in mind, if $w\in\mathfrak{S}_n$ is the unique permutation such that $\mathfrak{t}_{\mathtt{S}(i)}=\mathfrak{t}^{\lambda}\cdot w$, set $T_{\mathtt{S}(i)}=T_w$.  In order to better describe how $T_{\mathtt{S}}$ and the elements of $\mathscr{H}$ just defined interact with $u^+_{\mu}$, we also introduce the following notation.

\begin{itemize}
\item For any given multicomposition $\alpha$ of $n$, let $\overline{u}^+_{\alpha^{(i)}}$ refer to that aspect of $u^+_{\alpha}$ corresponding to the first $i$ components of $\alpha$:
\begin{displaymath}
\overline{u}^+_{\alpha^{(i)}}=\prod_{j=2}^{i+1}\prod_{k=1}^{|\alpha^{(1)}|+\cdots+|\alpha^{(j-1)}|}(L_k-Q_j)
\end{displaymath}
\item Similarly, we specify $\underline{u}^+_{\alpha^{(i)}}$ as those terms of $u^+_{\alpha}$ corresponding to the last $(r-1)-i$ components:
\begin{displaymath}
\underline{u}^+_{\alpha^{(i)}}=\prod_{j=i+1}^r\prod_{k=1}^{|\alpha^{(1)}|+\cdots+|\alpha^{(j-1)}|}(L_k-Q_j).
\end{displaymath}
\item
Finally, if $\gamma$ is another multicomposition of $n$, with $\alpha\unlhd\gamma$, let $u^+_{\gamma^{(i)}\setminus\alpha^{(i)}}$ signify the `difference' between $u^+_{\gamma}$ and $u^+_{\alpha}$ at the $i$-th component:
\begin{displaymath}
u^+_{\gamma^{(i)}\setminus\alpha^{(i)}}=\prod_{j=\overline{\alpha}^{(i)}+1}^{\overline{\gamma}^{(i)}}(L_j-Q_{i+1}).
\end{displaymath}
Note that when $\overline{\gamma}_{(i)}=\overline{\alpha}_{(i)}$ we set $u^+_{\gamma^{(i)}\setminus\alpha^{(i)}}=1$.
\end{itemize}
\begin{exm}
Suppose that $\alpha=((2,1),(2,2),(2,2,1))$.  Then 
\begin{align*}
\overline{u}^+_{\alpha^{(2)}}&=(L_1-Q_2)(L_2-Q_2)(L_3-Q_2)\\ &\times(L_1-Q_3)(L_2-Q_3)(L_3-Q_3)(L_4-Q_3)(L_5-Q_3)(L_6-Q_3)(L_7-Q_3)\\
\underline{u}^+_{\alpha^{(2)}}&=(L_1-Q_3)(L_2-Q_3)(L_3-Q_3)(L_4-Q_3)(L_5-Q_3)(L_6-Q_3)(L_7-Q_3),
\end{align*}
and, if $\gamma=((3,2,1),(3,1),(2))$, then
\begin{displaymath}
u^+_{\gamma^{(2)}\setminus\alpha^{(2)}}=(L_8-Q_3)(L_9-Q_3)(L_{10}-Q_3).
\end{displaymath}
\end{exm}

The following lemma expresses how the terms of $T_{\mathtt{S}}$ in a sense filter out the extraneous terms of $u^+_{\nu}$, to leave something of the form $u^+_\lambda h$ for some $\mathscr{H}$.  Subsequent results then establish that whatever is to the right of $u^+_\lambda$ after this filtering procedure involves a linear combination of the elements of $\mathbf{D}$ and $\mathbf{L}$ (multiplied on the right by elements of $h\in\mathscr{H}$).

\begin{lem}\label{form1}
Let $\lambda$ and $\nu$ be multipartitions of $n$ in $r$ parts with $\lambda\lhd\nu$.  If $\mathtt{S}$ is a semistandard $\nu$-tableau of type $\lambda$, then
 \begin{displaymath}
 T_{\mathtt{S}}u^+_{\nu}=\left(T_{\mathtt{S}(i)}\overline{u}^+_{\lambda^{(i)}}\underline{u}^+_{\nu^{(i+1)}}\right)u^+_{\nu^{(i)}\setminus\lambda^{(i)}}h_i
 \end{displaymath}
 for some $h_i\in\mathscr{H}$ and every $1\leq i\leq r-1$. 
 \end{lem} 
 \begin{proof}
 Let $w_\mathtt{S}$ be the permutation determining $T_{\mathtt{S}}$ and for $i$ with $1\leq i\leq r-1$ let $w_{\mathtt{S}(i)}$ be the unique permutation with $\mathfrak{t}_{\mathtt{S}(i)}=\mathfrak{t}^{\lambda}\cdot w_{\mathtt{S}(i)}$.  Furthermore, let $w_i$ be the permutation such that $\mathfrak{t}_{\mathtt{S}(i-1)}=\mathfrak{t}_{\mathtt{S}(i)}\cdot w_i$. Our strategy is to compare the lengths of $w_{\mathtt{S}}$ and $w_{\mathtt{S}(i)}w_i$ in order to show that the latter is reduced for all $i$ and that $T_{\omega_i}$ commutes `enough' with $u^+_{\mu}$ 

Recall that for $w\in\mathfrak{S}_n$ Dyer's coreflection cycle is the set
\begin{displaymath}
N(w)=\{(j,k)\in\mathfrak{S}_n:1\leq j<k\leq n\text{ and }jw>kw\}
\end{displaymath}  
and that the length $\ell(w)$ of $w$ is the same as the cardinality of $N(w)$.   Hence the length of $w_\mathtt{S}$ is equal to the number of pairs $(j,k)$ where $j<k$ and where $k$ occupies a node of $\mathfrak{t}_{\mathtt{S}}$ higher than that accommodating $j$.  We are now ready to begin the proof proper, which proceeds by induction.

Let $i=1$.  For any node $x\in[\lambda]$, let $n_{\mathtt{S}}(x)$ be  the number of entries in $\mathfrak{t}_{\mathtt{S}}$ less than $\mathfrak{t}_{\mathtt{S}}(x)$ and which are situated in lower nodes. 
Then
\begin{align*}
\ell(w_{\mathtt{S}})&=\sum_{x\in[\lambda^{(1)}]}n_{\mathtt{S}}(x)+\sum_{x\notin[\lambda^{(1)}]}n_{\mathtt{S}}(x)\\
&=\sum_{x\in[\lambda^{(1)}]}n_{\mathtt{S}}(x)+\ell(w_{\mathtt{S}(1)})
\end{align*} 
where the equality $\ell(w_{\mathtt{S}(1)})=\sum_{x\notin[\lambda^{(1)}]}n_{\mathtt{S}}(x)$ follows immediately from the construction of the $\mathfrak{t}_{\mathtt{S}(i)}$ tableaux for $i\geq 1$.   

Since $\sum_{x\in[\lambda^{(1)}]}n_{\mathtt{S}(x)}=\ell(w_1)$ we are done, since then we have $$w_{\mathtt{S}}=w_{\mathtt{S}(i)}\cdot w_1\phantom{iii}\text{ and }\phantom{iii}\ell(w_{\mathtt{S}})=\ell(w_{\mathtt{S}(1)})+\ell(w_1).$$
   Hence $w_{\mathtt{S}(1)}\cdot w_1$ is reduced and so $T_{\mathtt{S}}=T_{\mathtt{S}(1)}T_{w_1}$.

Recall that $w_1$ permutes only the set $\Omega_1=\{1,2,\ldots,|\nu^{(1)}|\}$, fixing all  other entries, such that the first component of $\mathfrak{t}_{\mathtt{S}}$ and $\mathfrak{t}^{\lambda}\cdot w_1$ are the same.  If $x\in[\lambda^{(1)}]$ and if $a$ is an entry less than $\mathfrak{t}_{\mathtt{S}}(x)$ but appearing in a node lower than $x$ in $\mathfrak{t}_{\mathtt{S}}$, then $aw_1>\mathfrak{t}_{\mathtt{S}}(x)w_1$.  The reverse implication is immediate, hence 
 $$\sum_{x\in[\lambda^{(1)}]}n_{\mathtt{S}}(x)=N(w_1)=\ell(w_1).$$  
 
 The general case of this part of our argument is almost exactly the same, with the only significant alteration being that for $1< i \leq r-1$ we note that $T_{\mathtt{S}(i)}$ fixes the first $i$ components of $\mathfrak{t}^{\lambda}$, and that $T_{w_i}$ fixes all entries not in the set $\{\overline{\lambda}_{(i)},\ldots,\overline{\nu}_{(i)}\}$, and adjust the proof for $i=1$ accordingly.

Since $w_1$ fixes all entries outside of $\Omega_1$, we see that $T_{w_1}$ commutes past $u^+_{\nu}$ to give us
\begin{align*}
T_{\mathtt{S}}u^+_{\nu}&=T_{\mathtt{S}(1)}u^+_{\nu}T_{w_1}\\&=\left(T_{\mathtt{S}(1)}\overline{u}^+_{\lambda^{(1)}}\underline{u}^+_{\nu^{(2)}}\right)u^+_{\nu^{(1)}\setminus\lambda^{(1)}}T_{w_1}.
\end{align*}

Now suppose the Lemma holds for some $i=k\leq r-1$.   Then
\begin{align*}
T_{\mathtt{S}}u^+_{\mu}=\left(T_{\mathtt{S}(k)}\overline{u}^+_{\lambda^{(k)}}\underline{u}^+_{\nu^{(k+1)}}\right)u^+_{\nu^{(k)}\setminus\lambda^{(k)}}h
\end{align*}
for some $h\in\mathscr{H}$.  But then $T_{\mathtt{S}(k)}=T_{\mathtt{S}(k+1)}T_{w_{k+1}}, $, where $w_{k+1}$ fixes all elements not in the set $\{\overline{\lambda}_{(k)}+1,\ldots,\overline{\nu}_{(k+1)}\}$ this allows $T_{w_{k+1}}$ to commute with $\overline{u}^+_{\lambda^{(k)}}\underline{u}^+_{\nu^{(k+1)}}$ and so
 \begin{displaymath}
 T_{\mathtt{S}}u^+_{\nu}=\left(T_{\mathtt{S}(k+1)}\overline{u}^+_{\lambda^{(k+1)}}\underline{u}^+_{\nu^{(k+2)}}\right)u^+_{\nu^{(k+1)}\setminus\lambda^{(k+1)}}T_{w_{k+1}}u^+_{\nu^{(k)}\setminus\lambda^{(k)}}h
 \end{displaymath}
 as required.  
   \end{proof}  
   
 \begin{exm}
 Let $\lambda=((2,1,1),(3,1),(2,1))$ and $\nu=((4,2),(3,1),(1))$, and let $\mathtt{S}\in\mathcal{T}_0(\nu,\lambda)$ be given by
\begin{displaymath}
\mathtt{S}=\left(\,\Yvcentermath1\young(\oneone\oneone\threeone\onethree,\twoone\twotwo)\,,\,\Yvcentermath1\young(\onetwo\onetwo\onetwo,\twothree)\,,\,\Yvcentermath1\young(\onethree)\,\right)
\end{displaymath}
Then
\begin{align*}
T_{\mathtt{S}}u^+_\nu&=T_{8,7,6,5}T_{8,7,6}T_{10}T_{3,4}\\ &\times\left((L_1-Q_2)\cdots (L_6-Q_2)\right)\left((L_1-Q_3)\cdots(L_{10}-Q_3)\right)\\
&=T_{8,7,6,5}T_{8,7,6}T_{10}\left((L_1-Q_2)\cdots(L_4-Q_2) \right)\\ 
&\times \left( (L_1-Q_3)\cdots (L_{10}-Q_3) \right)
\\&\times (L_5-Q_2)(L_6-Q_2)T_{3,4} \phantom{iiii}\left(=T_{\mathtt{S}(1)}\overline{u}^+_{\lambda^{(1)}}\underline{u}^+_{\nu^{(2)}}u^+_{\nu^{(1)}\setminus\lambda^{(1)}}h'\right)\\
&=T_{10}\left((L_1-Q_2)\cdots (L_4-Q_2) \right)\left((L_1-Q_3)\cdots (L_8-Q_2) \right)\\
&\times (L_9-Q_3)(L_{10}-Q_3)T_{8,7,6,5}T_{8,7,6} \\
&\times (L_5-Q_2)(L_6-Q_2)T_{3,4}\phantom{iiii}\left(=T_{\mathtt{S}(2)}\overline{u}^+_{\lambda^{(2)}}u^+_{\nu^{(2)}\setminus\lambda^{(2)}}h'' \right)
\end{align*}
 \end{exm}  
 It's worth remarking that in the above example, $T_{\mathtt{S}(r-1)}=T_{10}$ commutes past $u^+_{\lambda}(L_9-Q_3)$ and so, since $\mathfrak{l}^{(2)}=L_9-Q_3$ we can write $m_{\mathtt{S}\mathfrak{t}^\nu}=m_{\lambda}\mathfrak{l}^{(2)}h$ for some $h\in\mathscr{H}$.  We might well ask what happens when $|\lambda^{(r)}|= |\nu^{(r)}|$, since in this case $u^+_\nu$ does not contain a factor of $\mathfrak{l}^{(r-1)}$. 

To answer this question, suppose that $k$ is a positive integer and maximal such that $|\lambda^{(k)}|\neq |\nu^{(k)}|$.  If  $\mathtt{S}$ is a semistandard $\nu$-tableau of type $\lambda$, then each of the final $r-k$ components of $\mathtt{S}$ are a permutation of the corresponding component of the unique tableaux $\lambda(\mathfrak{t}^\nu)$.  Hence we can define another semistandard $\nu$-tableau of type $\lambda$, which we denote by $\mathtt{R}$ and form from $\mathtt{S}$ by replacing all the last $r-k$ components with those of $\lambda(\mathfrak{t}^{\nu})$.  
 \begin{cor}\label{Bunyip}
 Suppose that $\lambda$ and $\nu$ are both $r$-multicompositions of $n$ such that $\lambda\lhd\nu$ and suppose that $k$ is the maximal positive integer such that $|\lambda^{(k)}|\neq|\nu^{(k)}|$.  If $\mathtt{S}\in\mathcal{T}_0(\nu,\lambda)$, then 
\begin{displaymath}
x_\lambda T_\mathtt{S}u^+_\nu=x_\lambda T_{\mathtt{R}(k-1)}u^+_\lambda u^+_{\nu(k-1)\setminus\lambda(k-1)} h
\end{displaymath}
for some $h\in\mathscr{H}$.  Moreover, $u^+_{\nu(k-1)\setminus\lambda(k-1)}\neq1$.
 \end{cor}
 \begin{proof}
 If $|\lambda^{(l)}|=|\nu^{(l)}|$ for all $l\geq k$, then the permutation determining $T_{\mathtt{S}}$ stabilizes the  final $r-l$ components of $\mathfrak{t}^\nu$.  We may then write this permutation as a product of disjoint cycles $uv$ where $u\in\mathfrak{S}_X$ and $v\in\mathfrak{S}_Y$ for 
\begin{displaymath}
X=\left\{1,2,\ldots,\overline{\nu}_{(k)}\right\}\phantom{iii}\text{and}\phantom{iii}Y=\left\{ \overline{\nu}_{(k)}+1,\overline{\nu}_{(k)}+2,\ldots,n\right\}.
\end{displaymath}  
Hence $T_{\mathtt{S}}=T_{u}T_v$.  By definition, $T_u=T_\mathtt{R}$, where $\mathtt{R}$ is as in the discussion preceding this corollary.  Given that $\nu$ only permutes the entries of $Y$ within the components of $\mathfrak{t}^\nu$ they occupy, $T_v$ commutes with $u^+_\nu$ and we can write
\begin{displaymath}
x_\lambda T_\mathtt{S}u^+_\nu=x_\lambda T_\mathtt{R}u^+_\nu T_v.
\end{displaymath}

Applying Lemma \ref{form1} then yields 
\begin{displaymath}
x_\lambda T_\mathtt{S}u^+_\nu=x_\lambda T_{\mathtt{R}(k-1)}\overline{u}^+_{\lambda(k-1)} \underline{u}^+_{\nu(k)}u^+_{\lambda(k-1)\setminus\nu(k-1)}h
\end{displaymath} 
for $h\in\mathscr{H}$. However, by the definition of $k$, $\overline{u}^+_{\nu(k)}=\overline{u}^+_{\lambda(k)}$ and so 
\begin{displaymath}
x_\lambda T_\mathtt{S} u^+_\nu=\left(x_\lambda T_{\mathtt{R}(k-1)}u^+_\lambda\right)u^+_{\nu(k-1)\setminus\lambda(k-1)}h
\end{displaymath}
\end{proof}
Using the fact that $u^+_{\nu(k-1)\setminus\lambda(k-1)}=\mathfrak{l}^{(l-1)}h'$ for some $h'\in\mathscr{H}$, we may rewrite the conclusion of Corollary \ref{Bunyip} in the following, more suggestive form:
\begin{displaymath}
x_{\lambda}T_{\mathtt{S}}u^+_{\nu}=x_\lambda T_{\mathtt{R}(k-1)}u^+_{\lambda}\mathfrak{l}^{(k-1)}h'
\end{displaymath}
for some $h'\in\mathscr{H}$.  

We next show that $T_{\mathtt{R}(k-1)}$ can be expressed in a form more suited to our purposes.  For notational convenience, we will continue as if $k=r$ (in which case $\mathtt{R}(k-1)=\mathtt{S}(r-1))$.  The proofs may be easily adapted to the case where $k\neq r$ by replacing $\mathtt{S}(r-1)$ with $\mathtt{R}(k-1)$ and adjusting indices throughout accordingly.    

For $1\leq s\leq t$, define 
\begin{align*}
\pi(t,s)&=(s,s+1,\ldots,t)\\
\mathrm{D}(t,s)&=T_{\pi(t,s)}=T_{t-1}T_{t-2}\cdots T_s
\end{align*}

\begin{lem}[{\cite[Lemma 3.8]{Lyle2}}]
Suppose that $\mathtt{A}$ is an $\alpha$-tableau of type $\beta$ for compositions $\alpha$ and $\beta$ and write $\mathfrak{a}=\mathfrak{t}_{\mathtt{A}}$.  Let $\mathfrak{a}(0)=\mathfrak{a}$ and define $\mathfrak{a}(i)$ recursively by 
\begin{displaymath}
\mathfrak{a}(i)=\mathfrak{a}(i-1)\pi(i^*,i),
\end{displaymath}
where $i^*$ occupies the same node in $\mathfrak{a}(i-1)$ as does $i$ in $\mathfrak{a}$. Then
\begin{displaymath}
m_{\lambda}T_{\mathtt{A}}=m_{\lambda}\prod^{n-1}_{j=1}\mathrm{D}(j^*,j)
\end{displaymath}
\end{lem}

It's worth remarking that the lemma above was originally stated for partitions and row-semistandard tableaux.  However, the proof depends on neither of these facts, and so can be modified as we have done.

\begin{lem}\label{firstbit}
Let $\lambda$ and $\nu$ be multipartitions with $\lambda\lhd \nu$ such that there is at least one $1\leq i\leq r$ with $|\lambda^{(i)}|<|\nu^{(i)}|$.  Then either 
 \begin{displaymath}x_\lambda T_\mathtt{S}u^+_\nu=m_{\lambda}\mathfrak{l}^{(r-1)}h\phantom{iii} \text{ or }\phantom{iii}
 x_\lambda T_\mathtt{S}u^+_\nu=m_{\lambda}\mathrm{D}(x^*,x)\mathfrak{l}^{(r-1)}h'
 \end{displaymath}
 for some elements $h,h'\in\mathscr{H}$, and where $x=\overline{\lambda}_{(r)}+1$.
\end{lem}
\begin{proof}
Recall that the first $r-1$ components of  $\mathfrak{t}_{\mathtt{S}(r-1)}$ are identical to those of $\mathfrak{t}^{\lambda}$ and write $m_{\mathtt{S}\mathfrak{t}^{\nu}}$ as
\begin{displaymath}
x_{\lambda}T_{\mathtt{S}(r-1)}u^+_{\lambda}\mathfrak{l}^{(r-1)}h'.
\end{displaymath} 
If we regard $\nu$ and $\lambda$ as compositions via stacking, then  by the preceding lemma we have
\begin{displaymath}T_{\mathtt{S}(r-1)}=\prod^{\rho_{r}(\lambda)}_{i=\overline{\lambda}_{(r)}+1}\mathrm{D}(i^*,i)=\mathrm{D}\left(\left(\overline{\lambda}_{(r)}+1\right)^*,\overline{\lambda}_{(r)}+1\right)
\prod^{\rho_{r}(\lambda)}_{j=\overline{\lambda}_{(r)}+2}\mathrm{D}(j^*,j)\end{displaymath}

Clearly, the  product on the right hand side  commutes with $u^+_{\lambda}\mathfrak{l}^{(r-1)}$ in its entirety, since $j>\overline{\lambda}_{(r)}+1$ for every value of $j$ in the product.  There are now two possibilities to consider, either
\begin{displaymath}
\mathrm{D}\left(\left(\overline{\lambda}_{(r)}+1\right)^*,\overline{\lambda}_{(r)}+1\right)=1
\end{displaymath}
in which the first  part of the lemma holds, or   
\begin{displaymath}
\mathrm{D}\left(\left(\overline{\lambda}_{(r)}+1\right)^*,\overline{\lambda}_{(r)}+1\right)=T_{\left(\overline{\lambda}_{(r)}+1\right)^*-1}\cdots T_{\overline{\lambda}_{(r)}+2}T_{\overline{\lambda}_{(r)}+1}
\end{displaymath}
in which case it commutes past $u^+_{\lambda}$, but not $u^+_\lambda\mathfrak{l}^{(r-1)}$, providing the second part of the lemma. \end{proof}

Our task is then complete if $\mathrm{D}(x^*,x)$ can be expressed as a linear combination of $\mathfrak{d}^{(s)}_{d,t}$ and $\mathfrak{l}^{(s')}$ elements, each multiplied on the right  by elements of $h\in\mathscr{H}$.  Consider the following characterization of $\mathfrak{d}^{(s)}_{d,t}$, as found in \cite{Lyle2} (in the form of $h_{d,t}$ elements corresponding to partitions of $n$):  Let $m,a,b\geq 0$, and define
\begin{displaymath}
\langle m,a,b\rangle:=\left\{\mathbf{i}=(i_1,i_2,\ldots,i_b):m+1\leq i_1<i_2<\cdots<i_b\leq m+a+b\right\}
\end{displaymath}
Letting $(a,b)$ be a composition, we then have
 
\begin{displaymath}
\mathrm{C}(m,a,b)=\sum_{\mathbf{i}\in\langle m,a,b\rangle}\prod_{k=1}^b\mathrm{D}(m+a+k,i_k),
\end{displaymath}
and so 
\begin{equation}\label{thingamyjig}
\mathfrak{d}^{(s)}_{d,t}=\sum_{\mathbf{i}\in\langle \overline{\lambda}_{(d-1,s)},\lambda^{(s)}_d,t \rangle}\prod_{k=1}^t\mathrm{D}\left(\overline{\lambda}_{(d,s)}+k,i_k\right).
\end{equation}

\begin{lem}\label{secondbit}Let $x$ be as in the previous lemma.  Then there is some $l$ with $1\leq l\leq \rho_r(\lambda)$ such that
\begin{displaymath}
\mathrm{D}(x^*,x)u^+_\lambda\mathfrak{l}^{(r-1)}=u^+_\lambda\left(\mathfrak{l}^{(r-1)}h_0+\sum_{i=1}^{l-1}\mathfrak{d}^{(r)}_{i,1}h_{i} \right) 
\end{displaymath}
for some $h_0,h_1,\ldots, h_{l-1}\in\mathscr{H}$.
\end{lem}

\begin{proof}
Since the $\lambda$-tableau $\mathfrak{t}_{\mathfrak{s}(r-1)}$ defined previously are row standard and the first $r-1$ components are identical to those of $\mathfrak{t}^\lambda$ we have that $(\overline{\lambda}_{(r)}+1)^*$ appears at the beginning of a particular row of component $r$ .  Hence $\mathrm{D}(x^*,x)$ is of the form
\begin{displaymath}
\mathrm{D}\left(\overline{\lambda}_{(l-1,r)}+1,\overline{\lambda}_{(r)}+1 \right)
\end{displaymath} 
for some $l$ with $1\leq l\leq \rho_{r}(\lambda)$ (in fact, $l$ is the number of the row accomodating $(\overline{\lambda}_{(r)}+1)$.

We proceed by proving a slightly more general fact: namely that 
\begin{equation}\label{Induct}
\mathrm{D}\left(\overline{\lambda}_{(y-1,r)}+1,\overline{\lambda}_{(r)}+1 \right)u^+_\lambda\mathfrak{l}^{(r-1)}=u^+_{\lambda}\left(\mathfrak{l}^{(r-1)}h_0+\sum^{y-1}_{i=1}\mathfrak{d}^{(r)}_{i,1}h_i \right),
\end{equation}
where $h_0,h_1,\ldots,h_{l-1}$ (each depending on $y$), for all $y$ with $1\leq y$. 

The equality of \eqref{Induct} is trivially true when $y=1$, since in this case $\mathrm{D}(\overline{\lambda}_{(y-1,r)}+1,\overline{\lambda}_{(r)}+1)=1$.  Suppose instead that $y=2$.  Then, setting $s=r$ and $t=1$ in \eqref{thingamyjig} gives us 
 \begin{align*}
 \mathrm{D}\left(\overline{\lambda}_{(1,r)}+1,\overline{\lambda}_{(r)}+1 \right)&=T_{\overline{\lambda}_{(1,r)}}T_{\overline{\lambda}_{(1,r)-1}}\cdots T_{\overline{\lambda}_{(r)}+1}\\
&=\mathfrak{d}^{(r)}_{1,1}-\sum_{\overline{\lambda}_{(r)}+2}^{\overline{\lambda}_{1,r}+1} \mathrm{D}\left(\overline{\lambda}_{(1,r)}+1, i_1\right),
 \end{align*}
 and so 
 \begin{displaymath}
 \mathrm{D}\left(\overline{\lambda}_{(1,r)}+1,\overline{\lambda}_{(r)}+1 \right)u^+_{\lambda}\mathfrak{l}^{(r-1)}=u^+_{\lambda}\mathfrak{d}^{(r)}_{1,1}\mathfrak{l}^{(r-1)}-u^+_{\lambda}\mathfrak{l}^{(r-1)}\sum_{\overline{\lambda}_{(r)}+2}^{\overline{\lambda}_{1,r}+1} \mathrm{D}\left(\overline{\lambda}_{(1,r)}+1, i_1\right)
 \end{displaymath} 

Suppose now that the statement holds for some $y=k'$.  Then
\begin{align*}
\mathrm{D}\left(\overline{\lambda}_{(k'-1,r)}+1,\overline{\lambda}_{(r)}+1 \right)=\prod_{j=1}^{k'-1}\mathfrak{d}^{(r)}_{j,1}-\prod_{j=1}^{k'-2}\mathfrak{d}^{(r)}_{j,1}h_{k'-2}-\cdots-\mathfrak{d}^{(r)}_{1,1}h_{1}-h_0
\end{align*}
where $h_0,h_1,\ldots,h_{k'-2}\in\mathscr{H}$ all commute with $u^+_\lambda\mathfrak{l}^{(r)}$. Multiplying both sides by 
\begin{displaymath}
T_{\overline{\lambda}_{(k',r)}}T_{\overline{\lambda}_{(k',r)}-1}\cdots T_{\overline{\lambda}_{(k'-1,r)}+1}=\mathfrak{d}^{(r)}_{k',1}-\sum_{j'=\overline{\lambda}_{(k'-1,r)}+2}^{\overline{\lambda}_{(k',r)}+1}
\mathrm{D}\left(\overline{\lambda}_{(k',r)}+1,j'\right)
\end{displaymath}
shows us that $\mathrm{D}\left(\overline{\lambda}_{(k',r)}+1,\overline{\lambda}_{(k'-1,r)}+1 \right) \mathrm{D}\left(\overline{\lambda}_{(k'-1,r)}+1,\overline{\lambda}_{(r)}+1 \right)$ is equal to
\begin{align*}
\prod^{k'}_{j=1} \mathfrak{d}^{(r)}_{j,1}&-\prod^{k'-1}_{j=1}\mathfrak{d}^{(r)}_{j,1}\sum_{j'=\overline{\lambda}_{(k'-1,r)}+2}^{\overline{\lambda}_{(k',r)}+1}
\mathrm{D}\left(\overline{\lambda}_{(k',r)}+1,j'\right)\\
&-\prod^{k'-2}_{j=1}\mathfrak{d}^{(r)}_{j,1}\left(\mathfrak{d}^{(r)}_{k',1}-\sum_{j'=\overline{\lambda}_{(k'-1,r)}+2}^{\overline{\lambda}_{(k',r)}+1}
\mathrm{D}\left(\overline{\lambda}_{(k',r)}+1,j'\right)\right)h_{k'-2}\\\
&\vdots\\
&-\mathfrak{d}^{(r)}_{1,1}\left(\mathfrak{d}^{(r)}_{k',1}-\sum_{j'=\overline{\lambda}_{(k'-1,r)}+2}^{\overline{\lambda}_{(k',r)}+1}
\mathrm{D}\left(\overline{\lambda}_{(k',r)}+1,j'\right)\right)h_1\\
&-\left(\mathfrak{d}^{(r)}_{k',1}-\sum_{j'=\overline{\lambda}_{(k'-1,r)}+2}^{\overline{\lambda}_{(k',r)}+1}
\mathrm{D}\left(\overline{\lambda}_{(k',r)}+1,j'\right)\right)h_0.
\end{align*}

Observing that both
\begin{displaymath}
\mathfrak{d}^{(r)}_{k',1}\phantom{iiii}\text{and}\phantom{iiii}\sum_{j'=\overline{\lambda}_{(k'-1,r)}+2}^{\overline{\lambda}_{(k',r)}+1}
\mathrm{D}\left(\overline{\lambda}_{(k',r)}+1,j'\right)
\end{displaymath}
commute with $u^+_{\lambda}$ whenever $k'>1$,
 the statement of the lemma is then true for $y=k'+1$ and hence is true for all $y$ by induction.

Then all that needs to be done is to set $y=l$ in \eqref{Induct} and we are done. 
\end{proof}

\begin{cor}\label{fourthy}
Let $\lambda$ and $\mu$ be multipartitions of $n$ with $\lambda\lhd\nu$ and such that there is at least one $1\leq i\leq r$ with $|\lambda^{(i)}|<|\nu^{(i)}|$.  If $\mathtt{S}\in\mathcal{T}_0(\nu,\lambda)$,  then $m_{\mathtt{S}\mathfrak{t}^{\nu}}\in\mathfrak{I}$.
\end{cor}
\begin{proof}
The result follows at once from Lemma \ref{firstbit} and Lemma \ref{secondbit}. 
\end{proof}
We can now combine our results and finally prove Theorem \ref{Ideal1}: 
\begin{thmm}For every $r$-multipartition $\lambda$ of $n$
\begin{displaymath}
\mathfrak{I}=M^{\lambda}\cap\check{\mathscr{H}}^{\lambda}
\end{displaymath}
\end{thmm}
\begin{proof}
The result follows immediately from Lemma \ref{firsty}, Lemma \ref{secondy}, Lemma \ref{Thirdy}, and Corollary \ref{fourthy}.       
\end{proof}

\section{Constructing Homomorphisms: An Example}
Let $\lambda=((2,2),(2,1))$ and $\nu=((5),(2))$.  Then $\mathbf{D}=\{\mathfrak{d}^{(1)}_{1.1}, \mathfrak{d}^{(1)}_{2,1},\mathfrak{d}^{(2)}_{11}\}$ and $\mathbf{L}=\{\mathfrak{l}^{(1)}\}$, where
\begin{align*}
\mathfrak{d}^{(1)}_{1,1}&=1+T_2+T_{2,1} &\mathfrak{d}^{(1)}_{2,2}&=1+T_2+T_{2,1}+T_{2,3}+T_{2,3,1}+T_{2,3,1,2}\\
\mathfrak{d}^{(2)}_{1.1}&=1+T_6+T_{6,5} &\mathfrak{l}^{(1)}&=L_5-Q_2,
\end{align*}
and $\mathcal{T}_{0}(\nu,\lambda)=\{\mathtt{S}_1,\mathtt{S}_2\}$, where
\begin{align*}
\mathtt{S}_1=\left(\,\Yvcentermath1\young(\oneone\oneone\twoone\twoone\onetwo)\,,\,\Yvcentermath1\young(\onetwo\twotwo)\, \right) && \mathtt{S}_2&=\left(\,\Yvcentermath1\young(\oneone\oneone\twoone\twoone\twotwo)\,,\,\Yvcentermath1\young(\onetwo\onetwo)\,\right).
\end{align*}

The semistandard $\nu$-tableaux of type $\lambda$ determine homomorphisms from the permutation module $M^\lambda$ to the Specht module $M^\nu$ via
\begin{align*}
\theta_{\mathtt{S}}(m_\lambda h)=\left(\check{\mathscr{H}}^{\nu}+m_\nu(1+T_5)
\right)h, &&\theta_{\mathtt{S}_2}(m_\lambda h)=\left(\check{\mathscr{H}}^\nu+m_\nu T_{5,6}\right)h\end{align*}
for every $h\in\mathscr{H}$. We can use the results of this paper to attempt constructing a homomorphism $\hat{\theta}:S^{\lambda}\rightarrow S^{\nu}$ from any homomorphism $\theta:M^{\lambda}\rightarrow S^\nu$ given by
\begin{displaymath}
\theta=\alpha_1\theta_{\mathtt{S}_1}+\alpha_2\theta_{\mathtt{S}_2},
\end{displaymath}     
where $\alpha_1,\alpha_2\in\mathbb{F}$.  First, define  $\mathtt{A}_1, \mathtt{A}_2\in\mathcal{T}_0(\nu, [\lambda\cdot\mathfrak{d}^{(1)}_{1,1}])$, $\mathtt{A}_3,\mathtt{A}_4\in\mathcal{T}_0(\nu, [\mathfrak{d}^{(1)}_{2.1}\cdot\lambda])$, $\mathtt{A}_5\in\mathcal{T}_0(\nu,[\mathfrak{d}^{(2)}_{1,1}\cdot\lambda])$, and $\mathtt{B}\in\mathcal{T}_0(\nu,[\mathfrak{l}^{(1)}\cdot\lambda])$ by
\begin{align*}
\mathtt{A}_1&=\left(\,\Yvcentermath1\young(\oneone\oneone\oneone\twoone\onetwo)\,,\,\Yvcentermath1\young(\onetwo\twotwo)\, \right) &\mathtt{A}_2&=\left(\,\Yvcentermath1\young(\oneone\oneone\oneone\twoone\twotwo)\,,\,\Yvcentermath1\young(\onetwo\onetwo)\, \right)\\
\mathtt{A}_3&=\left(\,\Yvcentermath1\young(\oneone\oneone\oneone\oneone\onetwo)\,,\,\Yvcentermath1\young(\onetwo\twotwo)\,\right) &\mathtt{A}_4&=\left(\,\Yvcentermath1\young(\oneone\oneone\oneone\oneone\twotwo)\,,\,\Yvcentermath1\young(\onetwo\onetwo)\,\right)\\
\mathtt{A}_5&=\left(\,\Yvcentermath1\young(\oneone\oneone\twoone\twoone\onetwo)\,,\,\Yvcentermath1\young(\onetwo\onetwo)\, \right) &\mathtt{B}&=\left(\,\Yvcentermath1\young(\oneone\oneone\twoone\twoone\threeone)\,,\,\Yvcentermath1\young(\onetwo\twotwo)\,\right)
\end{align*}
Note that since these tableaux are all semistandard, the homomorphisms they determine are linearly independent.

Now we can act on $\theta$  by the elements of $\mathbf{D}$ and $\mathbf{L}$ in order to determine our conditions.  For instance
 \begin{align*}
 \theta_{\mathtt{S}_1}\left(m_\lambda\mathfrak{d}^{(1)}_{1,1}\right)&=\theta_{\mathtt{S}_1}(m_{\lambda})\mathfrak{d}^{(1)}_{1,1}=\check{\mathscr{H}}^\nu+m_\nu\left(1+T_2+T_{2,1}\right)(1+T_5)\\
 &=\check{\mathscr{H}}^{\nu}+\left(1+q+q^2\right)m_\nu(1+T_5)\\
 &=\left(1+q+q^2\right)\theta_{\mathtt{A}_1}\left(m_{\mathfrak{d}^{(1)}_{1,1}\cdot\lambda} \right).
 \end{align*}
 Performing the same calculation for every element of $\mathbf{D}$ yields
\begin{align*}
\theta\mathfrak{d}^{(1)}_{1,1}&=\alpha_1\left(1+q+q^2\right)\theta_{\mathtt{A}_1}+\alpha_2\left(1+q+q^2\right)\theta_{\mathtt{A}_2}\\
\theta\mathfrak{d}^{(1)}_{2,1}&=\alpha_1\left(1+q\right)\left(1+q+q^2\right)\theta_{\mathtt{A}_3}+\alpha_2(1+q)\left(1+q+q^2\right)\theta_{\mathtt{A}_4}\\
\theta\mathfrak{d}^{(2)}_{1,1}&=\left(\alpha_1(1+q)+\alpha_2q^2\right)\theta_{\mathtt{A}_5}.
\end{align*}

Setting each line equal to zero we see that $\alpha_1=\alpha_2=0$ whenever $e\neq3$, so let us assume that the contrary is true.  In this case $\theta\mathfrak{d}^{(1)}_{1.1}$ and $\theta\mathfrak{d}^{(1)}_{2,1}$ are zero, which leaves us with
the third line to satisfy.  Setting $\alpha_1=1$ we get \begin{equation}\label{example}
\theta\mathfrak{d}^{(2)}_{1,1}=0\Leftrightarrow \alpha_2=-q^{-2}(1+q).
\end{equation}This then leaves us with only $\theta\mathfrak{l}^{(1)}$ to contend with. 

Let $i$ be some positive integer, at most $n$.  The residue of $i$ in $\mathfrak{t}^\nu$ is given by 
\begin{displaymath}
\mathrm{res}_{\mathfrak{t}^\nu}(i)=q^{y-x}Q_z
\end{displaymath}
where $\mathfrak{t}^\nu(x,y,z)=i$. Using this definition and the fact that $m_\nu L_i=\mathrm{res}_\gamma(i)m_\nu$ (an easy consequence of \cite[Proposition 3.7]{JM1}), and 
\begin{displaymath}T_iL_i=L_{i+1}T_i-(q-1)L_{i+1}
\end{displaymath}
(see, for instance, \cite[Lemma 13.2]{Ariki1}) we have
 \begin{equation}\label{Example2}
 \theta\mathfrak{l}^{(1)}=\left(\alpha_1\left(q^4Q_1-qQ_2\right)-\alpha_2q(q-1)Q_2\right)\theta_{\mathtt{B}}.
 \end{equation}
   
We can then substitute the conclusion of \eqref{example} into \eqref{Example2} to yield
 \begin{displaymath}
 \theta\mathfrak{l}^{(1)}=\left(q^4Q_1-q^{-1}Q_2\right)\theta_{\mathtt{B}}
 \end{displaymath}
 Thus the homomorphism $\theta:M^\lambda\rightarrow S^\nu$ defined by $\alpha_1$ and $\alpha_2$ given above satisfies the conditions of Theorem \ref{Ideal1} whenever $q^4Q=q^{-1}Q_2$ and $e=3$.

\end{document}